\documentclass{article}

\usepackage{amstext}
\usepackage{amssymb}
\usepackage{amsfonts}
\usepackage{amsmath}
\usepackage{amsthm}

\usepackage{stmaryrd}

\usepackage{enumitem}

\usepackage{bussproofs}

\usepackage{tikz-cd}
\tikzset{
  symbol/.style={
    draw=none,
    every to/.append style={
      edge node={node [sloped, allow upside down, auto=false]{$#1$}}}
  }
}

\usepackage{comment}

\usepackage{xcolor} 

\theoremstyle{plain}

\newtheorem{Theorem}{Theorem}

\newtheorem{theorem}[Theorem]{Theorem}
\newtheorem{definition}[Theorem]{Definition}
\newtheorem{lemma}[Theorem]{Lemma}

\def\squareforqed{\hbox{\rlap{$\sqcap$}$\sqcup$}}
\def\qed{\ifmmode\squareforqed\else{\unskip\nobreak\hfil
\penalty50\hskip1em\null\nobreak\hfil\squareforqed
\parfillskip=0pt\finalhyphendemerits=0\endgraf}\fi}

\def\cal#1{{\mathcal #1}}

\def\bpi#1{\mathbf{\Pi}^0_{#1}}

\def\QCB{{{\bf QCB}_0}}
\def\IN{{\mathbb N}}

\def\sierp{{\mathbb S}}

\def\I#1{{{\mathrm{\mathbf I}}(#1)}}

\def\PO{{\mathrm{\bold O}}}

\def\calC{{\cal C}}
\def\calD{{\cal D}}
\def\calF{{\cal F}}
\def\calG{{\cal G}}

\def\calL{{\cal L}}

\def\calS{{\cal S}}
\def\calT{{\cal T}}

\def\calX{{\cal X}}

\newcommand{\ODS}{{\mathsf{ODS}}}
\newcommand{\CHS}{{\mathsf{CHS}}}
\def\QCB{{{\bf QCB}_0}}

\newcommand{\y}{{\mathsf{\bf d}}}

\newcommand{\Mor}{\mathsf{Mor}}
\newcommand{\Obj}{\mathsf{Obj}}

\def\int{{\mathrm{\mathbf{int}}}}

\def\wayabovearrow{\rlap{\raise-.25ex\hbox{$\shortuparrow$}}\raise.25ex\hbox{$\shortuparrow$}}
\def\waybelowarrow{\rlap{\raise.25ex\hbox{$\shortdownarrow$}}\raise-.25ex\hbox{$\shortdownarrow$}}

\def\token#1{\lceil#1\rceil}

\makeatletter
\DeclareFontFamily{OMX}{MnSymbolE}{}
\DeclareSymbolFont{MnLargeSymbols}{OMX}{MnSymbolE}{m}{n}
\SetSymbolFont{MnLargeSymbols}{bold}{OMX}{MnSymbolE}{b}{n}
\DeclareFontShape{OMX}{MnSymbolE}{m}{n}{
    <-6>  MnSymbolE5
   <6-7>  MnSymbolE6
   <7-8>  MnSymbolE7
   <8-9>  MnSymbolE8
   <9-10> MnSymbolE9
  <10-12> MnSymbolE10
  <12->   MnSymbolE12
}{}
\DeclareFontShape{OMX}{MnSymbolE}{b}{n}{
    <-6>  MnSymbolE-Bold5
   <6-7>  MnSymbolE-Bold6
   <7-8>  MnSymbolE-Bold7
   <8-9>  MnSymbolE-Bold8
   <9-10> MnSymbolE-Bold9
  <10-12> MnSymbolE-Bold10
  <12->   MnSymbolE-Bold12
}{}

\let\llangle\@undefined
\let\rrangle\@undefined
\DeclareMathDelimiter{\llangle}{\mathopen}%
                     {MnLargeSymbols}{'164}{MnLargeSymbols}{'164}
\DeclareMathDelimiter{\rrangle}{\mathclose}%
                     {MnLargeSymbols}{'171}{MnLargeSymbols}{'171}
\makeatother

\def\name#1{{\llangle{#1}\rrangle}}

\begin{document}


\title{On the completeness of the countable fragment of geometric logic}

\author{Matthew de Brecht}
\date{Graduate School of Human and Environmental Studies\\ Kyoto University, Japan\\\texttt{matthew@i.h.kyoto-u.ac.jp}}

\maketitle

\begin{abstract}
We give self contained proofs of the completeness of countable $\sigma$-coherent theories, for both propositional and predicate logic. By $\sigma$-coherent logic we mean the fragment of geometric logic that only allows countable signatures and countably infinite disjunctions. In the propositional case, each theory determines a quasi-Polish space of models, and formulas (up to provable equivalence) are interpreted as continuous functions from the space of models to the Sierpinski space.  In the predicate case, each theory determines a quasi-Polish category of models, and formulas (up to provable equivalence) are interpreted as continuous functors from the category of models to the quasi-Polish category of overt discrete quasi-Polish spaces. We further extend this analogy by investigating a notion of ``sobriety'' for certain topological categories.
\end{abstract}


\section{Introduction}

This paper is a continuation of \cite{dbr26} that aims to clarify its connections with $\sigma$-coherent logic, which is the fragment of geometric logic that only allows countable signatures and countably infinite disjunctions. 

In Section~2 we give a self contained proof of a completeness theorem for countable propositional $\sigma$-coherent theories, which is equivalent to the spatiality of countably presented frames. This result is originally due to Fourman and Grayson \cite{FG82}, and proved again later by Heckmann \cite{H15} using very different methods. Our approach is closer to Heckmann's, but presented in a logical framework. To briefly summarize, given a countable propositional $\sigma$-coherent theory $\calT$, there is a quasi-Polish space $X_\calT$ of models of $\calT$, such that logical formulas $\varphi$ can be identified (up to $\calT$-provable equivalence of formulas) with continuous functions $\llbracket \varphi\rrbracket \colon X_\calT \to \sierp$, where $\sierp$ is the Sierpinski space. From this perspective, the spatiality result says that the function space $\sierp^{X_\calT}$ is isomorphic to the frame presented by $\calT$.

Then in Section~3 we give a self-contained completeness proof for countable predicate $\sigma$-coherent theories, which already follows from Chen's conceptual completeness result \cite{Chen19} (see also \cite{Chen25}) and Lopez-Escobar's completeness theorem \cite{LE65}. We present the result using (topologically) continuous functors\footnote{As in \cite{dbr26}, we use \emph{continuous functor} in the topological sense, and do not use the term in the categorical sense of preserving limits.} instead of  \'{e}tale spaces (based on the equivalence result in \cite{dbr26}), and our proof of the completeness theorem is a generalization of Heckmann's proof. To briefly summarize, given a countable predicate $\sigma$-coherent theory $\calT$, there is a quasi-Polish category $\calX_\calT$ of models of $\calT$, such that logical formulas $\varphi$ can be identified (up to $\calT$-provable equivalence of formulas) with continuous functors $\llbracket \varphi\rrbracket \colon \calX_\calT \to \ODS$, where $\ODS$ is the quasi-Polish category of overt discrete quasi-Polish spaces. From this perspective Chen's conceptual completeness result says that the functor category $\ODS^{\calX_\calT}$ is equivalent to the $\sigma$-pretopos presented by $\calT$. 

In Section~4, we formulate well-known applications to classical first-order logic, Heyting-valued models, and forcing within our framework. In Section~5 we extend the analogy between frames and $\sigma$-pretoposes by defining ``sobriety'' for certain topological categories.

We assume the reader has read \cite{dbr26}, and we will use the same notation and definitions here.\footnote{The only exception is that in \cite{dbr26} we used the terms \emph{$\omega_1$-coherent logic} and \emph{$\omega_1$-pretopos}, whereas here we will use the terms \emph{$\sigma$-coherent logic} and \emph{$\sigma$-pretopos}. } We refer to Johnstone \cite{J02} and Vickers \cite{V07} for references on geometric logic and topos theory, and refer to \cite{dbr,dbr20} for information on quasi-Polish spaces.

\section{Propositional $\sigma$-coherent Logic}

In this section we show that the frame of opens of quasi-Polish spaces correspond to the Lindenbaum algebras of countable propostional $\sigma$-coherent theories. This is essentially equivalent to the fact that countably presented frames are spatial, which was first shown by Fourman and Grayson \cite{FG82} and again shown by Heckmann \cite{H15} with a different proof using the Baire categegory theorem. Our proof closely follows Heckmann's proof, except we present it within a logical framework rather than an algebraic one.

\subsection{Syntax and inference rules}

Let $V$ be a countable set of propositional variables. The \emph{propositional $\sigma$-coherent formulas} (or just \emph{formulas}) over $V$ are defined inductively as follows:
\begin{itemize}
\item
$\bot$  is a formula.
\item
$\top$ is a formula.
\item
If $p\in V$ then $p$ is a formula.
\item
If $\varphi$ and $\psi$ are formulas then $\varphi\wedge \psi$ is a formula.
\item
If $\varphi_i$ is a formula for each $i\in \IN$, then $\bigvee_{i\in \IN} \varphi_i$ is a formula.
\end{itemize}
For simplicity we do not include binary joins, but $\varphi\vee\psi$ can be represented by the formula $\bigvee_{i\in\IN} \varphi_i$ with $\varphi_0 = \varphi$ and $\varphi_1 = \psi$ and $\varphi_i = \bot$ for $i>1$. For a finite set $F=\{\varphi_1,\ldots,\varphi_n\}$ of formulas, we write $\bigwedge F$ as an abbreviation for $\varphi_1 \wedge \cdots \wedge \varphi_n$. If $F$ is the empty set, then we identify $\bigwedge F$ with $\top$.

Below are the inference rules for propositional $\sigma$-coherent logic (see D1.3.1 of \cite{J02} or Section~2.2 of \cite{V07}):
\begin{itemize}
\item
$\varphi \vdash \varphi$ (identity),\hspace{0.2cm}
\AxiomC{$\varphi \vdash \psi$}
\AxiomC{$\psi \vdash \chi$}
\RightLabel{(cut)}
\BinaryInfC{$\varphi \vdash \chi$}
\DisplayProof
\item
$\varphi\wedge \bigvee_{i\in \IN} \psi_i \vdash \bigvee_{i\in \IN} (\varphi\wedge\psi_i)$ (distributivity)
\item
$\varphi\vdash \top$,\hspace{0.2cm} $\varphi\wedge\psi \vdash \varphi$,\hspace{0.2cm} $\varphi\wedge\psi \vdash \psi$,\hspace{0.2cm}
\AxiomC{$\chi \vdash \varphi$}
\AxiomC{$\chi \vdash \psi$}
\RightLabel{($\wedge I$)}
\BinaryInfC{$\chi \vdash \varphi\wedge \psi$}
\DisplayProof
\item
$\bot\vdash\psi$,\hspace{0.2cm} $\varphi_j\vdash \bigvee_{i\in \IN} \varphi_i$ (for $j\in \IN$),\hspace{0.2cm}
\AxiomC{$\varphi_i \vdash \psi$ (for each $i\in \IN$)}
\RightLabel{($\vee I$)}
\UnaryInfC{$\bigvee_{i\in I} \varphi_i \vdash \psi$}
\DisplayProof
\end{itemize}

A sequent $\varphi\vdash\psi$ is \emph{valid} if it can be derived from the above inference rules. We write $\varphi\not\vdash\psi$ if the sequent is \emph{not} valid. Here is a sample derivation showing the validity of distributivity in the other direction: 
\begin{center}
\AxiomC{$\varphi\wedge\psi_j \vdash\varphi$}
\AxiomC{$\varphi\wedge\psi_j \vdash \psi_j$}
\AxiomC{$\psi_j\vdash\bigvee_{i\in I} \psi_i$}
\RightLabel{(cut)}
\BinaryInfC{$\varphi\wedge\psi_j \vdash\bigvee_{i\in I} \psi_i$} 
\RightLabel{($\wedge I$)}
\BinaryInfC{$\varphi\wedge\psi_j \vdash\varphi\wedge\bigvee_{i\in I} \psi_i$} 
\RightLabel{($\vee I$)}
\UnaryInfC{$\bigvee_{i\in I} (\varphi\wedge\psi_i) \vdash\varphi\wedge\bigvee_{i\in I} \psi_i$} 
\DisplayProof
\end{center}

A countable set $\calT$ of sequents of the form $\varphi\vdash \psi$ (with variables from a fixed countable set $V$) is called a \emph{countable propositional $\sigma$-coherent theory}, and the elements of $\calT$ are called \emph{axioms}. We say that $\varphi\vdash \psi$ is a \emph{theorem of $\calT$}, or \emph{derivable from $\calT$}, if there is a derivation of $\varphi\vdash \psi$ from the inference rules and the axioms in $\calT$. We write $\varphi\vdash^{\calT} \psi$ to denote a theorem of $\calT$.

A \emph{basic formula}\footnote{These are called Horn formulas in \cite{J02}.} is a formula of the form $\bigwedge F$, where $F\subseteq V$ is finite (if $F$ is empty then we identify $\bigwedge F$ with $\top$). Since we assume $V$ is countable, there are only countably many basic formulas. It will often be convenient to assume formulas are in a ``normal form" in the following sense.

\begin{lemma}[\cite{J02} Lemma~D1.3.8]\label{lem:normalform}
Every propositional $\sigma$-coherent formula is provably equivalent to a formula of the form $\bigvee_{i\in\IN}\rho_i$, where each $\rho_i$ is either a basic formula or $\bot$.
\qed
\end{lemma}

\subsection{Space of models}

For topological spaces $X$ and $Y$, we write $Y^X$ for the space of all continuous functions $f\colon X\to Y$, equipped with the compact-open topology. For the spaces we use in this paper, this agrees with the topology on the exponential object $Y^X$ in the cartesian closed category $\QCB$ \cite{BSS07}.

Let $V$ be a countable set of variables, and $\sierp=\{\bot,\top\}$ Sierpinski space. We write $\sierp^V$ for the space of continuous functions from $V$ to $\sierp$, where we consider $V$ to have the discrete topology.

For $M\in \sierp^V$, we define the satisfaction relation $M\vDash \varphi$ inductively as follows:
\begin{enumerate}
\item
$M\vDash \bot$ never holds.
\item
$M\vDash \top$ always holds.
\item
If $p\in V$, then $M\vDash p$ if and only if $M(p)=\top$.
\item
$M\vDash \varphi\wedge \psi$ if and only if $M\vDash \varphi$ and $M\vDash \psi$.
\item
$M\vDash \bigvee_{i\in \IN} \varphi_i$ if and only if $M \vDash \varphi_i$ for some $i\in \IN$.
\end{enumerate}

Note that $\llbracket \varphi \rrbracket = \{ M\in\sierp^V\mid M\vDash \varphi\}$ is an open subset of $\sierp^V$, and sets of the form $\llbracket p \rrbracket$ for $p\in V$ form a subbasis for $\sierp^V$. Furthermore, the open sets $\llbracket\rho\rrbracket$ (for basic formulas $\rho$) form a basis for the topology on $\sierp^V$. We will sometimes write $\neg\llbracket\varphi\rrbracket$ for the complement of $\llbracket\varphi\rrbracket$, which is a closed subset of $\sierp^V$.

Let $\calT$ be a countable propositional $\sigma$-coherent theory with variables from the countable set $V$. The \emph{space of $\calT$-models} is the subspace $X_\calT \subseteq \sierp^V$ defined as:
\[ X_\calT=\bigcap\{ \neg\llbracket\varphi\rrbracket\cup\llbracket\psi\rrbracket \mid \varphi\vdash \psi\text{ is an axiom of } \calT\}.\]
It is clear that $X_\calT$ is quasi-Polish, and conversely every quasi-Polish space is homeomorphic to the space of models of some countable theory.

We write $\llbracket \varphi \rrbracket^{\calT} = \{ M\in X_\calT\mid M\vDash \varphi\}$ for the restriction of $\llbracket \varphi \rrbracket$ to $X_\calT$. Note that $\llbracket \varphi \rrbracket^{\calT}$ is open in $X_\calT$, and conversely every open subset of $X_\calT$ is equal to $\llbracket \varphi \rrbracket^{\calT}$ for a suitable formula $\varphi$.

\subsection{Completeness}

Fix a countable set of variables $V$ throughout this section.

\begin{lemma}\label{lem:axiomless_proofs}
A sequent $\varphi \vdash \psi$ is valid if and only if $\llbracket \varphi \rrbracket \subseteq \llbracket \psi \rrbracket$ in $\sierp^V$.
\end{lemma}
\begin{proof}
If $\varphi \vdash \psi$ is valid then $\llbracket \varphi \rrbracket\subseteq \llbracket\psi\rrbracket$ can be shown by induction on the derivation.

Conversely, assume $\varphi \not\vdash\psi$. Without loss of generality, we can assume $\varphi = \bigvee_{i\in \IN}\bigwedge F_i$ and $\psi = \bigvee_{j\in \IN}\bigwedge G_j$ are in normal form. Then for some $i\in\IN$ we must have $\bigwedge F_i \not\vdash \psi$, since otherwise $(\vee I)$ could be applied to obtain $\varphi\vdash \psi$. Define $M\colon V\to \sierp$ as $M(p)=\top \iff p\in F_i$. Clearly, $M\vDash \varphi$. If there was $k\in\IN$ with $G_k \subseteq F_i$, then we could derive $\bigwedge F_i \vdash \bigwedge G_k$, but since $\bigwedge G_k \vdash \bigvee_{j\in \IN}\bigwedge G_j$ we could apply cut to obtain $\bigwedge F_i \vdash \psi$, a contradiction. Therefore, $M\not\vDash \bigwedge G_j$ for each $j\in\IN$, hence $M\not\vDash \psi$. Therefore, $M\in\llbracket\varphi\rrbracket \setminus \llbracket\psi\rrbracket$, hence $\llbracket\varphi\rrbracket\not\subseteq\llbracket\psi\rrbracket$.
\end{proof}

The following lemma corresponds to Proposition~3.10 of Heckmann \cite{H15}.

\begin{lemma}\label{lem:singleaxiom_proofs}
The sequent $\varphi \vdash \psi$ is derivable from the single axiom $\sigma\vdash\theta$ if and only if $\varphi \vdash \sigma \vee \psi$ and $\varphi \wedge \theta \vdash \psi$ are both valid.
\end{lemma}
\begin{proof}
Assume $\varphi \vdash \sigma \vee \psi$ is not valid. Then Lemma~\ref{lem:axiomless_proofs} implies $\llbracket\varphi\rrbracket\not\subseteq\llbracket\sigma\vee\psi\rrbracket$. Fix $M\in\sierp^V$ such that $M\vDash \varphi$ but $M\not\vDash \sigma\vee\psi$. Then $M\not\vDash \sigma$, hence $M$ is a model of the axiom $\sigma\vdash\theta$ but $M$ is not a model of $\varphi\vdash\psi$. A similar argument shows that if $\varphi \wedge\theta \vdash \psi$ is not valid then there is a model of $\sigma\vdash\theta$ that does not satisfy $\varphi\vdash\psi$. It follows from the soundness of the proof system that $\varphi \vdash \psi$ is not derivable from the single axiom $\sigma\vdash\theta$.

For the converse, assume $\varphi \vdash \sigma\vee\psi$ and $\varphi\wedge\theta\vdash\psi$ are both valid. Then:
\begin{center}
\AxiomC{$\varphi\vdash\varphi$}
\AxiomC{$\vdots$}
\UnaryInfC{$\varphi \vdash \sigma\vee\psi$}
\BinaryInfC{$\varphi\vdash\varphi\wedge( \sigma\vee\psi)$}
\AxiomC{$\varphi\wedge( \sigma\vee\psi)\vdash (\varphi\wedge\sigma)\vee(\varphi\wedge\psi)$}
\BinaryInfC{$\varphi\vdash (\varphi\wedge\sigma)\vee(\varphi\wedge\psi)$}
\DisplayProof
\end{center}
Furthermore:
\begin{center}
\AxiomC{$\varphi\wedge\sigma\vdash \varphi$}
\AxiomC{$\varphi\wedge\sigma\vdash \sigma$}
\AxiomC{$\sigma\vdash\theta$}
\BinaryInfC{$\varphi\wedge\sigma \vdash \theta$}
\BinaryInfC{$\varphi\wedge\sigma \vdash \varphi\wedge\theta$}
\AxiomC{$\vdots$}
\UnaryInfC{$\varphi\wedge\theta\vdash\psi$}
\BinaryInfC{$\varphi\wedge\sigma \vdash \psi$}
\AxiomC{$\varphi\wedge\psi\vdash\psi$}
\BinaryInfC{$(\varphi\wedge\sigma)\vee(\varphi\wedge\psi)\vdash\psi$}
\DisplayProof
\end{center}
Finally, apply cut to obtain a derivation of $\varphi\vdash\psi$ from the axiom $\sigma\vdash \theta$.
\end{proof}

The following completeness theorem for countable propositional $\sigma$-coherent theories was first shown by Fourman and Grayson \cite{FG82}, and it also corresponds to the spatiality theorem of Heckmann (Theorem~3.13 of \cite{H15}).

\begin{theorem}\label{thrm:prop_complete}
Let $\calT$ be a countable propositional $\sigma$-coherent theory and $\varphi$, $\psi$ propositional $\sigma$-coherent formulas. Then $\varphi \vdash^{\calT}\psi$ if and only if $\llbracket \varphi\rrbracket^\calT \subseteq \llbracket \psi\rrbracket^\calT$.
\end{theorem}
\begin{proof}
If $\varphi \vdash\psi$ is derivable from $\calT$, then a proof by induction on the derivation shows that $\llbracket\varphi\rrbracket\cap X_\calT \subseteq \llbracket\psi\rrbracket$.

Conversely, assume that $\varphi \vdash\psi$ is not derivable from $\calT$. Let $(\sigma_i \vdash \theta_i)_{i\in\IN}$ be an enumeration of $\calT$. Let $(\rho_j)_{j\in\IN}$ be an enumeration of all basic formulas satisfying $\rho_j \vdash^{\calT} \psi$, and define
\[ \chi = \bigvee_{j\in\IN} \rho_j.\]
By applying $(\vee I)$ we obtain $\chi \vdash^{\calT} \psi$. Furthermore, assuming without loss of generality $\psi = \bigvee_{k\in\IN} \vartheta_k$ for basic formulas $\vartheta_k$, we have that $\vartheta_k \vdash \psi$ is valid for each $k\in\IN$, which implies each $\vartheta_k$ is equal to some $\rho_j$. Therefore, $\psi \vdash \chi$ is valid.

Clearly, $\llbracket\varphi\rrbracket\not\subseteq \llbracket\chi\rrbracket$, since otherwise Lemma~\ref{lem:axiomless_proofs} would imply $\varphi\vdash\chi$, but then cut would yield $\varphi\vdash^{\calT} \psi$, a contradiction.

Fix $i\in\IN$. Assume $\rho$ is a basic formula such that $\llbracket\rho\rrbracket\cap\neg\llbracket\chi\rrbracket\not=\emptyset$. If $\llbracket\rho\rrbracket\cap\neg \llbracket\chi\rrbracket \cap (\neg\llbracket\sigma_i\rrbracket\cup\llbracket\theta_i\rrbracket)=\emptyset$ then $\llbracket\rho\rrbracket\subseteq \llbracket\sigma_i\rrbracket\cup \llbracket\chi\rrbracket$ and $\llbracket\rho\rrbracket\cap \llbracket\theta_i\rrbracket \subseteq \llbracket\chi\rrbracket$, hence Lemmas~\ref{lem:axiomless_proofs} and \ref{lem:singleaxiom_proofs} imply $\rho\vdash^{\calT} \chi$. But then $\rho\vdash^{\calT} \psi$, hence $\rho = \rho_j$ for some $j\in\IN$. But then $\rho\vdash \chi$, which contradicts $\llbracket\rho\rrbracket\cap\neg\llbracket\chi\rrbracket\not=\emptyset$.

Thus $\neg\llbracket\sigma_i\rrbracket\cup\llbracket\theta_i\rrbracket$ is dense in $\neg\llbracket\chi\rrbracket$ for each $i\in\IN$, hence by the (generalized) Baire Category Theorem\footnote{This is called the \emph{extended local Baire property} by R.~Heckmann (see Corollary~2.25 of Heckmann \cite{H15}). See also Theorem~3.14 of \cite{BG15} and Lemma~5.1 of \cite{d1?}.}
\[X_\calT = \bigcap_{i\in\IN} (\neg\llbracket\sigma_i\rrbracket\cup\llbracket\theta_i\rrbracket)\]
is dense in $\neg\llbracket\chi\rrbracket$. Since $\llbracket\varphi\rrbracket\cap\neg\llbracket\chi\rrbracket\not=\emptyset$ there exists $M\in \llbracket\varphi\rrbracket\cap\neg\llbracket\chi\rrbracket\cap X_\calT$. Clearly, $M\in X_\calT$ and $M\vDash \varphi$ but $M\not\vDash\psi$ because  $\llbracket \psi\rrbracket \subseteq \llbracket \chi\rrbracket$.
\end{proof}

Let $\sierp^{X_\calT}$ be the space of continuous functions from $X_\calT$, which is isomorphic to the frame of opens of $X_\calT$. Since each element of $\sierp^{X_\calT}$ is equal to $\llbracket\varphi\rrbracket^\calT$ for some propositional $\sigma$-coherent formula $\varphi$, the above theorem shows that $\sierp^{X_\calT}$ is isomorphic to the set of propositional $\sigma$-coherent formulas modulo $\calT$-provable equivalence. In other words, $\sierp^{X_\calT}$ is the Lindenbaum algebra of $\calT$.

\section{Predicate $\sigma$-coherent Logic}

This section extends the completeness result of the previous section to predicate logic. To accomplish this, rather than a quasi-Polish \emph{space} of models, we will need a quasi-Polish \emph{category} of models. For simplicity, we assume a single sort, and a relational language (i.e., no function symbols). However this is not really a restriction (see D1.4.9 and D1.4.13 of \cite{J02}).

\subsection{Syntax and inference rules}

We assume a countable supply of variables. A finite list of variables $\vec{x}$, in which each variable occurs only once, is called a \emph{context}. We write $x\in\vec{x}$ if the variable $x$ occurs in $\vec{x}$.  We write $\vec{x},y$ for the extension of $\vec{x}$ with a variable $y$, and $\vec{x},\vec{y}$ for the concatenation of two lists. The length of a list is denoted $len(\vec{x})$. If $\vec{y}$ has the same length as $\vec{x}$, then we say $\vec{y}$ is \emph{compatible} with $\vec{x}$.

\subsubsection{Formulas}

A \emph{language} is a countable set $\calL$ of predicate symbols, each with a given arity. The \emph{predicate $\sigma$-coherent formulas} (or just \emph{formulas}) over $\calL$ are defined inductively as follows:
\begin{itemize}
\item
$\vec{x}.\bot$  is a formula.
\item
$\vec{x}.\top$ is a formula.
\item
If  $x_i,x_j \in \vec{x}$ then $\vec{x}.x_i=x_j$ is a formula.
\item
If $x_{i_k}$ are variables in $\vec{x}$ ($1\leq k \leq n$) and $P\in\calL$ is an $n$-ary predicate symbol, then $\vec{x}.P(x_{i_1},\ldots,x_{i_n})$ is a formula.
\item
If $\vec{x}.\varphi$ and $\vec{x}.\psi$ are formulas then $\vec{x}.(\varphi\wedge \psi)$ is a formula.
\item
If $\vec{x}.\varphi_i$ is a formula for each $i\in \IN$, then $\vec{x}.\bigvee_{i\in \IN} \varphi_i$ is a formula.
\item
If $\vec{x},y. \varphi$ is a formula then $\vec{x}.(\exists y)\varphi$ is a formula.
\end{itemize}

Note that infinitary disjunctions have only finitely many free variables. We often abbreviate $(\exists y_1)\cdots(\exists y_n)\varphi$ as $(\exists \vec{y})\varphi$.

A \emph{basic formula} is a formula of the form $\vec{x}.\bigwedge F$, where $F$ is a finite set of formulas of the form $x=y$ or $P(x_{i_1},\ldots,x_{i_n})$ (again we identify the empty meet with $\top$). Since $\calL$ is countable there are only countably many basic formulas.

A normal form lemma holds for predicate $\sigma$-coherent logic as well.

\begin{lemma}[\cite{J02} Lemma~D1.3.8]\label{lem:normalform_pred}
Every predicate $\sigma$-coherent formula is provably equivalent to a formula of the form $\vec{x}.\bigvee_{i\in\IN}(\exists \vec{y_i}) \rho_i$, where each $\rho_i$ is either a basic formula or $\bot$.
\qed
\end{lemma}

\subsubsection{Inference rules}

The inferences rules for predicate $\sigma$-coherent logic are listed below (see D1.3.1 of \cite{J02} and Sections 2.2 and 3.1 of \cite{V07}). As in \cite{J02}, we do not distinguish between formulas that differ only in the names of bound variables (i.e., $\alpha$-equivalent formulas).

\begin{itemize}
\item
$\varphi \vdash_{\vec{x}} \varphi$ (identity),\hspace{0.2cm}
\AxiomC{$\varphi \vdash_{\vec{x}} \psi$}
\AxiomC{$\psi \vdash_{\vec{x}} \chi$}
\RightLabel{(cut)}
\BinaryInfC{$\varphi \vdash_{\vec{x}} \chi$}
\DisplayProof
\item
$\varphi\wedge \bigvee_{i\in I} \psi_i \vdash_{\vec{x}} \bigvee_{i\in \IN} (\varphi\wedge\psi_i)$ (distributivity)
\item
$\varphi\vdash_{\vec{x}} \top$,\hspace{0.2cm} $\varphi\wedge\psi \vdash_{\vec{x}}\varphi$,\hspace{0.2cm} $\varphi\wedge\psi \vdash_{\vec{x}}\psi$,\hspace{0.2cm}
\AxiomC{$\chi \vdash_{\vec{x}} \varphi$}
\AxiomC{$\chi \vdash_{\vec{x}} \psi$}
\RightLabel{($\wedge I$)}
\BinaryInfC{$\chi \vdash_{\vec{x}} \varphi\wedge \psi$}
\DisplayProof
\item
$\bot\vdash\psi$,\hspace{0.2cm} $\varphi_j\vdash_{\vec{x}} \bigvee_{i\in I} \varphi_i$ (if $j\in I$),\hspace{0.2cm}
\AxiomC{$\varphi_i \vdash_{\vec{x}} \psi$ (for each $i\in \IN$)}
\RightLabel{($\vee I$)}
\UnaryInfC{$\bigvee_{i\in \IN} \varphi_i \vdash_{\vec{x}} \psi$}
\DisplayProof
\end{itemize}

\begin{itemize}
\item
\AxiomC{$\varphi \vdash_{\vec{x}} \psi$}
\RightLabel{(substitution)}
\UnaryInfC{$\varphi[\vec{s}/\vec{x}] \vdash_{\vec{y}} \psi[\vec{s}/\vec{x}]$}
\DisplayProof

where $\vec{s}$ is a list of variables that is compatible with $\vec{x}$, and $\vec{y}$ is any context that includes all of the variables in $\vec{s}$.
\item
$\top\vdash_{\vec{x}} x=x$,\hspace{0.2cm} $(\vec{x}=\vec{y})\wedge\varphi \vdash_{\vec{z}} \varphi[\vec{y}/\vec{x}]$

where $\vec{z}$ includes all the variables in $\vec{x}$ and $\vec{y}$ and those free in $\varphi$.
\item
\AxiomC{$\varphi \vdash_{\vec{x},y} \psi$}
\UnaryInfC{$(\exists y)\varphi \vdash_{\vec{x}} \psi$}
\DisplayProof, \hspace{0.2cm}
\AxiomC{$(\exists y)\varphi \vdash_{\vec{x}} \psi$}
\UnaryInfC{$\varphi \vdash_{\vec{x},y} \psi$}
\DisplayProof
\item
$\varphi \wedge (\exists y)\psi \vdash_{\vec{x}} (\exists y) (\varphi\wedge\psi)$ (Frobenius rule)
\end{itemize}

A sequent $\varphi \vdash_{\vec{x}}\psi$ is \emph{valid} if it can be derived from the above inference rules. We write $\varphi \not\vdash_{\vec{x}}\psi$ if the sequent is \emph{not} valid. A countable set $\calT$ of sequents of the form $\varphi\vdash_{\vec{x}} \psi$ is called a \emph{countable predicate $\sigma$-coherent theory}, and the elements of $\calT$ are called \emph{axioms}  (there is an implicit universal quantification over the variables in $\vec{x}$ for each axiom $\varphi \vdash_{\vec{x}} \psi$). We say that $\varphi\vdash_{\vec{x}} \psi$ is a \emph{theorem of $\calT$}, or \emph{derivable from $\calT$}, if there is a derivation of $\varphi\vdash_{\vec{x}} \psi$ from the inference rules and the axioms in $\calT$. We write $\varphi\vdash_{\vec{x}}^{\calT} \psi$ to denote a theorem of $\calT$.

\subsection{Category of Models}\label{subsec_CatOfMods}

The quasi-Polish category of \emph{overt discrete quasi-Polish spaces} ($\ODS$) is defined as in \cite{dbr26}. Briefly, objects $A\in \ODS_\Obj$ are defined as partial equivalence relations (PER) $\equiv_A$ on $\IN$, and morphisms $f\colon A\to B$ in $\ODS_\Mor$ are functions from the equivalence classes of $\equiv_A$ to the equivalence classes of $\equiv_B$. The topology on $\ODS_\Obj$ is generated by subbasic opens of the form $\{ A \in\ODS_\Obj \mid a \equiv_A a'\}$ for $a,a'\in\IN$. The topology on $\ODS_\Mor$ is generated by subbasic opens of the form $\{ f\colon A\to B \mid f(a)=b \,\&\, a' \equiv_A a'' \,\&\, b' \equiv_B b'' \}$ for $a,a',a'', b, b', b'' \in\IN$.

The specialization order on $\ODS_\Obj$ is given as $A \leq B$ if and only if $a \equiv_A a'$ implies $a \equiv_B a'$. $A$ is a \emph{subobject} of $B$ if $A \leq B$ and for all $a,a'$ if $a \equiv_A a$ and $a \equiv_B a'$ then $a \equiv_A a'$. This is different than how subobjects are defined in category theory, but it is easy to see that there is a bijection between subobjects in our sense and category theoretical subobjects in $\ODS$. In particular, if $f\colon A \to B$ is a monomorphism, then $f$ can be factored uniquely as an isomorphism $i\colon A \to A'$ to a subobject $A' \leq B$, followed by the embedding $A' \to B$ (we will sometimes write this embedding as $A' \leq B$).

Let $\calL$ be a countable language, which we view as a space with the discrete topology. The space of $\calL$-structures is defined as the subspace \[(\calS_\calL)_\Obj \subseteq \ODS_\Obj\times (\ODS_\Obj)^\calL\]
of pairs $M=(\overline{M},\llbracket \cdot\rrbracket_M)$, where $\llbracket \cdot\rrbracket_M\colon\calL\to\ODS_\Obj$ assigns to each $n$-ary predicate symbol $P\in\calL$ a subobject $\llbracket P\rrbracket_M\leq \overline{M}^n$.  We write $\equiv_M$ for the partial equivalence relation (PER) corresponding to $M$, and abbreviate $a \equiv_M a$ as $a\in M$. The specialization order on $\ODS_\Obj$ is $\bpi 2$, and for each $n\in\IN$ there is a continuous function $(\cdot)^n \colon\ODS_\Obj\to\ODS_\Obj$ which maps an object to its $n$-fold product, hence $(\calS_\calL)_\Obj$ is a quasi-Polish space. 

Given $M,M'\in (\calS_\calL)_\Obj$, a \emph{homomorphism} $h\colon M\to M'$ is a morphism $\overline{h}\colon \overline{M}\to\overline{M}'$ in $\ODS$ such that $\llbracket P\rrbracket_M(a_1,\ldots,a_n)$ implies $\llbracket P\rrbracket_{M'}(\overline{h}(a_1),\ldots,\overline{h}(a_n))$ for each $n$-ary predicate symbol $P\in\calL$. Here we are writing $\llbracket P\rrbracket_M(a_1,\ldots,a_n)$ as an abbreviation for $\langle a_1,\ldots,a_n \rangle \in \llbracket P\rrbracket_M$,  which recall is an abbreviation for $\langle a_1,\ldots,a_n \rangle \equiv_{\llbracket P\rrbracket_M} \langle a_1,\ldots,a_n \rangle$.

Define $(\calS_\calL)_\Mor\subseteq \ODS_\Mor\times (\calS_\calL)_\Obj\times (\calS_\calL)_\Obj$ to be the subspace of tuples $(h,M,M')$ such that $h\colon M\to M'$ is a homomorphism. It is easy to see that $(\calS_\calL)_\Mor$ is a quasi-Polish space. 

Together, $(\calS_\calL)_\Obj$ and $(\calS_\calL)_\Mor$ determine a quasi-Polish category $\calS_\calL$, which we call the \emph{category of (countable) $\calL$-structures}.

We write $\IN^*$ for the set of finite sequences of natural numbers. A list $\vec{a}\in\IN^*$ is \emph{compatible} with a context $\vec{x}$ (or a formula $\vec{x}.\varphi$) if $len(\vec{a})=len(\vec{x})$. The satisfaction relation $M\vDash (\vec{x}.\varphi)(\vec{a})$ for a predicate $\sigma$-coherent formula $\vec{x}.\varphi$ and a compatible list $\vec{a}\in\IN^*$ is defined inductively as follows:
\begin{enumerate}
\item
$M\vDash (\vec{x}.\bot)(\vec{a})$ never holds.
\item
$M\vDash (\vec{x}.\top)(\vec{a})$ if and only if $a_k\in M$ for each $a_k\in\vec{a}$.
\item
$M\vDash (\vec{x}.x_i=x_j)(\vec{a})$ if and only if $a_i\equiv_M a_j$ and $M\vDash (\vec{x}.\top)(\vec{a})$.
\item
$M\vDash (\vec{x}.P(x_{i_1},\ldots,x_{i_n}))(\vec{a})$ if and only if $\llbracket P\rrbracket_M(a_{i_1},\ldots,a_{i_n})$ and $M\vDash (\vec{x}.\top)(\vec{a})$.
\item
$M\vDash (\vec{x}.\varphi\wedge \psi)(\vec{a})$ if and only if $M\vDash (\vec{x}.\varphi)(\vec{a})$ and $M\vDash (\vec{x}.\psi)(\vec{a})$.
\item
$M\vDash (\vec{x}.\bigvee_{i\in \IN}\varphi_i)(\vec{a})$ if and only if $M\vDash (\vec{x}.\varphi_i)(\vec{a})$ for some $i\in \IN$.
\item
$M\vDash (\vec{x}.(\exists y)\varphi)(\vec{a})$ if and only if $M\vDash (\vec{x},y.\varphi)(\vec{a},b)$ for some $b \in M$.
\end{enumerate}
In some cases we will omit the context and write $M\vDash \varphi(\vec{a})$ when it does not cause confusion. Note that a basis for the topology on $(\calS_\calL)_\Obj$ is given by sets of the form
\[\name{ \rho(\vec{a})} = \{ M \mid M\vDash\vec{x}.\rho(\vec{a})\},\]
where $\vec{x}.\rho$ is a basic formula and $\vec{a}\in\IN^*$ is compatible with $\vec{x}$. 

Given a formula in context $\vec{x}.\varphi$, define the functor $\llbracket \vec{x}.\varphi \rrbracket\colon\calS_\calL \to \ODS$ as  \[\llbracket\vec{x}.\varphi \rrbracket_\Obj(M) = \{ \langle a_1,\ldots,a_n\rangle \in M^n \mid M \vDash \vec{x}.\varphi(a_0,\ldots,a_n) \}\]
and let $\llbracket\vec{x}.\varphi\rrbracket_\Mor(h\colon M\to M')$ be the map sending $\langle a_0,\ldots,a_n\rangle \in \llbracket\vec{x}.\varphi\rrbracket_\Obj(M)$ to $\langle h(a_0),\ldots, h(a_n)\rangle\in \llbracket\vec{x}.\varphi\rrbracket_\Obj(M')$. Note that $\llbracket\vec{x}.\varphi \rrbracket$ is a continuous functor.

Let $\calT$ be a countable predicate $\sigma$-coherent theory. $M\in (\calS_\calL)_\Obj$ is a \emph{$\calT$-model} if $\llbracket \vec{x}.\varphi \rrbracket_\Obj(M) \subseteq \llbracket \vec{x}.\psi \rrbracket_\Obj(M)$ for each axiom $\varphi \vdash_{\vec{x}} \psi$ in $\calT$. We write $\calX_\calT$ for the full subcategory of $\calS_\calL$ of (countable) models of $\calT$. It is easy to see that $\calX_\calT$ is a quasi-Polish category.

We write $\llbracket \vec{x}.\varphi \rrbracket^{\calT}$ for the restriction of the functor $\llbracket \vec{x}.\varphi \rrbracket$ to $\calX_\calT$. Thus each formula corresponds to a continuous functor from $\calX_\calT$ to $\ODS$. A converse can be obtained by interpreting each continuous functor from $\calX_\calT$ to $\ODS$ as an \emph{imaginary sort} (see Theorem~8.1 of \cite{Chen19}).

\subsection{Completeness}

\begin{lemma}\label{lem:minimal_struct}
Let $\vec{x}.\rho$ be a basic formula, and let $\vec{a}\in\IN^*$ be compatible with $\vec{x}$. Set $\name{ \rho(\vec{a})} = \{ M \mid M\vDash\vec{x}.\rho(\vec{a})\}$, which is a basic open subset of $(\calS_\calL)_\Obj$. Then there is $M_0\in \name{ \rho(\vec{a})} $ that is minimal with respect to the specialization order on  $(\calS_{\calL})_\Obj$.

Furthermore, if $\rho \vdash_{\vec{x}} x_i = x_j$ is valid whenever $a_i = a_j$, then for any formula $\vec{y}.\theta$ and $\vec{b}\in\IN^*$ compatible with $\vec{y}$, we have $M_0\vDash \theta(\vec{b})$ if and only if there is a list of variables $\vec{s}$ which is compatible with $\vec{y}$ and such that the following both hold:
\begin{enumerate}
\item
for each $i<len(\vec{s})$ there is $j<len(\vec{x})$ such that $s_i = x_j$ and $b_i = a_j$.
\item
$\rho \vdash_{\vec{x}} \theta[\vec{s}/\vec{y}]$ is valid.
\end{enumerate}
\end{lemma}
\begin{proof}
$M_0$ is defined as follows. The PER for $M_0$ is only defined on the elements in $\vec{a}$, and is defined as $a_i\equiv_{M_0} a_j$ if and only if $\rho\vdash_{\vec{x}} x_i = x_j$ is valid. Similarly, for each predicate symbol $P\in\calL$, define $\llbracket P \rrbracket_{M_0}(a_{i_0},\ldots,a_{i_n})$ to hold if and only if $\rho\vdash_{\vec{x}} P(x_{i_0},\ldots,x_{i_n})$ is valid. 

It easily follows that if $\vec{x}.\vartheta$ is a basic formula, then $M_0 \vDash \vartheta(\vec{a})$ if and only if $\rho\vdash_{\vec{x}} \vartheta$ is valid. In particular, $M_0\vDash\rho(\vec{a})$ hence $M_0\in \name{ \rho(\vec{a})}$.

Fix any $M\in \name{ \rho(\vec{a})} $. Then $M\vDash \rho(\vec{a})$, hence if $\rho\vdash_{\vec{x}} x_i = x_j$ is valid then the soundness of the inference rules implies $a_i \equiv_M a_j$, and if $\rho\vdash_{\vec{x}} P(x_{i_0},\ldots,x_{i_n})$ is valid then $\llbracket P \rrbracket_M(a_{i_0},\ldots,a_{i_n})$ holds. Therefore, $M_0\leq M$. 

Assume $M_0\vDash \theta(\vec{b})$. Without loss of generality, $\theta = \bigvee_{k\in \IN}(\exists\vec{z_k})\vartheta_k$, where each $\vartheta_k$ is a basic formula. Let $k\in\IN$ be such that $M_0\vDash (\exists\vec{z_k})\vartheta_k(\vec{b})$. Let $\vec{c}$ be compatible with $\vec{z_k}$ such that $M_0\vDash \vartheta_k(\vec{b},\vec{c})$. From the definition of $M_0$, each $b_i$ must be equal to some $a_j$, so there is a list $\vec{s}$ such that for each $i$ there is $j$ with $s_i = x_j$ and $b_i = a_j$. Fix a similar list $\vec{t}$ for $\vec{c}$. Then $M_0\vDash \vartheta_k[(\vec{s},\vec{t})/(\vec{y},\vec{z_k})](\vec{a})$, which implies $\rho \vdash_{\vec{x}}\vartheta_k[(\vec{s},\vec{t})/(\vec{y},\vec{z_k})]$ is valid. We also have the following:
\begin{center}
\AxiomC{$(\exists \vec{z_k})\vartheta_k \vdash_{\vec{y}} \theta$}
\UnaryInfC{$\vartheta_k \vdash_{\vec{y},\vec{z_k}} \theta$}
\RightLabel{(substitution)}
\UnaryInfC{$\vartheta_k[(\vec{s},\vec{t})/(\vec{y},\vec{z_k})] \vdash_{\vec{x}} \theta[\vec{s}/\vec{y}]$}
\DisplayProof
\end{center}
The variables $\vec{z_k}$ are bounded in $\theta$, so the substituion $\vec{t}/\vec{z_k}$ has no effect on $\theta$ and so it has been abbreviated in the last line above. By applying cut we obtain $\rho \vdash_{\vec{x}} \theta[\vec{s}/\vec{y}]$. 
\end{proof}

In practice, the assumption on $\rho$ in the second part of the above lemma is not very strict, because in most applications we can either replace $\rho$ with the basic formula $\rho \wedge \bigwedge\{x_i = x_j \mid a_i=a_j\}$, or we can choose $\vec{a}\in\IN^*$ that does not contain duplicates and make the requirement vacuous. Later we will use a combination of these two approaches.

\begin{lemma}\label{lem:axiomless_proofs_pred}
The sequent $\varphi \vdash_{\vec{x}} \psi$ is valid if and only if $\llbracket\vec{x}.\varphi \rrbracket_\Obj(M)\subseteq \llbracket \vec{x}.\psi \rrbracket_\Obj(M)$ for each $M\in\calS_\calL$.
\end{lemma}
\begin{proof}
If $\varphi \vdash_{\vec{x}} \psi$ and $M \vDash \vec{x}.\varphi(\vec{a})$, then $M \vDash \vec{x}.\psi(\vec{a})$ can be shown by induction on the derivation. Therefore, $\llbracket \vec{x}.\varphi \rrbracket_\Obj(M)\subseteq \llbracket\vec{x}.\psi \rrbracket_\Obj(M)$.

Conversely, assume $\varphi \not\vdash_{\vec{x}} \psi$. We can assume $\varphi = \bigvee_{i\in \IN}(\exists\vec{y_i})\rho_i$, where each $\rho_i$ is a basic formula. For some $i\in\IN$ we must have $\rho_i \not\vdash_{\vec{x},\vec{y_i}}  \psi$, since otherwise the $\exists$-rule and $(\vee I)$ could be applied to obtain $\varphi \vdash_{\vec{x}} \psi$. Let $M_0$ be the minimal structure in $\name{\rho_i(\vec{a},\vec{b})}$, where $\vec{a}=\{0,\ldots,len(\vec{x})-1\}$ and $\vec{b}=\{len(\vec{x}),\ldots,len(\vec{x})+len(\vec{y_i})-1\}$. Then $M_0 \not\vDash \psi(\vec{a})$, because otherwise $\rho_i\vdash_{\vec{x},\vec{y_i}} \psi$ would be valid by the minimality of $M_0$ (we can ignore the substitution by $\vec{s}$ because our choice of $\vec{a},\vec{b}$ prevents it from changing any variables). Therefore, $\vec{a} \in \llbracket\vec{x}.\varphi \rrbracket_\Obj(M_0)$ and $\vec{a}\not\in \llbracket\vec{x}.\psi \rrbracket_\Obj(M_0)$, hence $\llbracket \vec{x}.\varphi \rrbracket_\Obj(M_0)\not\subseteq \llbracket \vec{x}.\psi \rrbracket_\Obj(M_0)$.
\end{proof}


\begin{lemma}\label{lem:singleaxiom_proofs_pred2}
Let $\vec{x},\vec{z}.\rho$ be a basic formula, and let $\vec{x}.\psi$, $\vec{y}.\sigma$, and $\vec{y}.\theta$ be arbitrary formulas. Assume $\vec{a},\vec{b},\vec{c}\in\IN^*$ are compatible with $\vec{x}$, $\vec{y}$, and $\vec{z}$, respectively. Further assume that $\rho \vdash_{\vec{x},\vec{z}} x_i = x_j$ whenever $a_i = a_j$, that $\vec{a}$ and $\vec{c}$ are disjoint, and that $\vec{c}$ does not contain any duplicates.

If $\name{\rho(\vec{a},\vec{c})} \subseteq \name{\sigma(\vec{b})} \cup \name{\psi(\vec{a})}$ and $\name{\rho(\vec{a},\vec{c})} \cap \name{\theta(\vec{b})} \subseteq \name{\psi(\vec{a})}$ then $\rho \vdash_{\vec{x},\vec{z}} \psi$ is derivable from the single axiom $\sigma \vdash_{\vec{y}} \theta$.
\end{lemma}
\begin{proof}
Assume the conditions of the lemma hold. By Lemma~\ref{lem:minimal_struct} there is a minimal structure $M_0\in \name{\rho(\vec{a},\vec{c})}$. If $M_0 \vDash \psi(\vec{a})$, then minimality and the assumptions on $\rho$, $\vec{a}$, and $\vec{c}$ would imply that $\rho \vdash_{\vec{x},\vec{z}} \psi$ is valid, and we are done.

So we can assume $M_0 \vDash \sigma(\vec{b})$. By minimality, there is a list of variables $\vec{s}$ from $\vec{x},\vec{z}$ that is compatible with $\vec{y}$ such that 
\begin{enumerate}
\item
for each $i<len(\vec{s})$, either there is $j<len(\vec{x})$ such that $s_i = x_j$ and $b_i = a_j$, or there is $j<len(\vec{z})$ such that $s_i = z_j$ and $b_i = c_j$.
\item
$\rho\vdash_{\vec{x},\vec{z}} \sigma[\vec{s}/\vec{y}]$ is valid.
\end{enumerate}
For notational convenience, we set $\hat{\sigma} = \sigma[\vec{s}/\vec{y}]$ and $\hat{\theta} = \theta[\vec{s}/\vec{y}]$.

We next show that $\rho \wedge \hat{\theta} \vdash_{\vec{x},\vec{z}} \psi$ is valid. This is trivial if $\hat{\theta} \vdash_{\vec{x},\vec{z}} \bot$, so we can assume $\hat{\theta} = \bigvee_{k\in \IN}(\exists\vec{w_k})\vartheta_k$, where each $\vartheta_k$ is a basic formula.

Fix $k$ and let $M'_0$ be minimal such that $M'_0\vDash (\rho\wedge\vartheta_k)(\vec{a},\vec{c},\vec{e})$, where $\vec{e}\in\IN^*$ is compatible with $\vec{w_k}$, is disjoint from $\vec{a},\vec{c}$, and does not contain duplicates. Then $M'_0 \vDash \rho(\vec{a},\vec{c})$ and $M'_0\vDash \hat{\theta}(\vec{a},\vec{c})$, and since $\name{\hat{\theta}(\vec{a},\vec{c})} \subseteq \name{\theta(\vec{b})}$ our assumption implies $M'_0\vDash \psi(\vec{a})$. The minimality of $M'_0$ and the disjointness of the contexts implies $\rho\wedge\vartheta_k \vdash_{\vec{x},\vec{z},\vec{w_k}}\psi$ is valid, from which we can derive $\rho\wedge(\exists w_k)\vartheta_k \vdash_{\vec{x},\vec{z}}\psi$ by using the Frobenius rule.

Since $k$ was arbitrary, it follows that $\bigvee_{k\in\IN} (\rho\wedge(\exists w_k)\vartheta_k) \vdash_{\vec{x},\vec{z}}\psi$ is valid. Using distributivity and the Frobenius rule again, we obtain a derivation of $\rho \wedge \hat{\theta} \vdash_{\vec{x},\vec{z}} \psi$.

Thus $\rho \vdash_{\vec{x},\vec{z}} \hat{\sigma}$ and $\rho \wedge \hat{\theta} \vdash_{\vec{x},\vec{z}} \psi$ are both valid, so we have the following derivation:
\begin{center}
\AxiomC{$\rho \vdash_{\vec{x},\vec{z}} \rho$}
\AxiomC{$\vdots$}
\UnaryInfC{$\rho \vdash_{\vec{x},\vec{z}} \hat{\sigma}$}
\AxiomC{$\sigma\vdash_{\vec{y}} \theta$}
\UnaryInfC{$\hat{\sigma}  \vdash_{\vec{x},\vec{z}} \hat{\theta} $}
\BinaryInfC{$\rho \vdash_{\vec{x},\vec{z}}\hat{\theta}$}
\BinaryInfC{$\rho \vdash_{\vec{x},\vec{z}} \rho \wedge \hat{\theta}$}
\AxiomC{$\vdots$}
\UnaryInfC{$\rho \wedge \hat{\theta} \vdash_{\vec{x},\vec{z}} \psi$}
\BinaryInfC{$\rho\vdash_{\vec{x},\vec{z}}\psi$}
\DisplayProof
\end{center}
Therefore, $\rho\vdash_{\vec{x},\vec{z}}\psi$ is derivable from the single axiom $\sigma \vdash_{\vec{y}} \theta$. 
\end{proof}


Given $\vec{x}.\varphi$ and $\vec{x}.\psi$, define $\llbracket \vec{x}.\varphi \rrbracket^\calT \leq \llbracket\vec{x}.\psi \rrbracket^\calT$ if and only if $\llbracket \vec{x}.\varphi \rrbracket^\calT_\Obj(M)\subseteq \llbracket \vec{x}.\psi \rrbracket^\calT_\Obj(M)$ for each $M\in \calX_\calT$. The next theorem shows that the functor category $\ODS^{\calX_\calT}$ can be viewed as (a generalization of) the Lindenbaum algebra of $\calT$ (and the analogy becomes precise if we allow the imaginary sorts).

\begin{theorem}
Let $\calT$ be a countable predicate $\sigma$-coherent theory. Then $\varphi \vdash_{\vec{x}}^\calT \psi$ if and only if $\llbracket \vec{x}.\varphi \rrbracket^\calT \leq \llbracket\vec{x}.\psi \rrbracket^\calT$.
\end{theorem}
\begin{proof}
If $\varphi \vdash_{\vec{x}}^\calT \psi$ is valid, then a proof by induction on the derivation shows that $\llbracket \vec{x}.\varphi \rrbracket_\Obj(M)\subseteq \llbracket \vec{x}.\psi \rrbracket_\Obj(M)$ for each $M\in\calX_\calT$.


Conversely, assume that $\varphi \vdash^{\calT}_{\vec{x}}\psi$ is not derivable. Let $(\sigma_i \vdash_{\vec{y}_i}\theta_i)_{i\in\IN}$ be an enumeration of $\calT$. Set 
\[S = \{ (\exists \vec{z})\rho\mid \text{$\vec{x},\vec{z}.\rho$ is a basic formula such that $(\exists \vec{z})\rho \vdash^{\calT}_{\vec{x}} \psi$}\}\]
and define
\[\chi = \bigvee_{\vartheta\in S}  \vartheta.\]
As in the proof of Theorem~\ref{thrm:prop_complete}, we have $\psi\vdash_{\vec{x}}\chi$ is valid and $\chi \vdash^{\calT}_{\vec{x}} \psi$ but $\varphi \not\vdash_{\vec{x}} \chi$. It follows from Lemma~\ref{lem:axiomless_proofs_pred} that there is $\vec{a}$ such that $\name{\varphi(\vec{a})} \cap \neg\name{\chi(\vec{a})} \not=\emptyset$.

We show that $\neg\name{\sigma_i(\vec{b})} \cup \name{\theta_i(\vec{b})}$ is dense in  $\neg\name{\chi(\vec{a})}$ for each $i\in\IN$ and $\vec{b}\in\IN^*$ that is compatible with $y_i$. Let $U$ be any open set that intersects $\neg\name{\chi(\vec{a})}$. Let $\vec{x},\vec{z}.\rho$ be a basic formula and $\vec{c}$ (disjoint from $\vec{a}$) such that the basic open $\name{\rho(\vec{a},\vec{c})}$ is contained in $U$ and intersects $\neg\name{\chi(\vec{a})}$. We can assume that $\rho$ implies $x_i = x_j$ whenever $a_i = a_j$, and that $\vec{c}$ does not contain duplicates.

Assume for a contradiction that $\name{\rho(\vec{a},\vec{c})}$ does not intersect $\neg\name{\chi(\vec{a})}\cap(\neg\name{\sigma_i(\vec{b})} \cup \name{\theta_i(\vec{b})})$. Then $\name{\rho(\vec{a},\vec{c})} \subseteq \name{\sigma(\vec{b})} \cup \name{\chi(\vec{a})}$ and $\name{\rho(\vec{a},\vec{c})} \cap \name{\theta(\vec{b})} \subseteq \name{\chi(\vec{a})}$, hence Lemma~\ref{lem:singleaxiom_proofs_pred2} applies and we obtain $\rho\vdash^{\calT}_{\vec{x},\vec{z}}\chi$. But then $\rho\in S$, and we obtain a contradiction with the assumption that $\name{\rho(\vec{a},\vec{c})}$ intersects $\neg\name{\chi(\vec{a})}$.

It follows that $\neg\name{\sigma_i(\vec{b})} \cup \name{\theta_i(\vec{b})}$ is dense in  $\neg\name{\chi(\vec{a})}$ (for each $i\in\IN$ and compatible $\vec{b}$), hence by the (generalized) Baire category theorem, 
\[(\calX_\calT)_\Obj = \bigcap_{i\in\IN} \bigcap_{\vec{b}}  (\neg\name{\sigma_i(\vec{b})} \cup \name{\theta_i(\vec{b})})\]
is dense in $\neg\name{\chi(\vec{a})}$. Since $\name{\varphi(\vec{a})} \cap \neg\name{\chi(\vec{a})} \not=\emptyset$, there is $M \in (\calX_\calT)_\Obj \cap \name{\varphi(\vec{a})} \cap \neg\name{\chi(\vec{a})}$. Since $\neg\name{\chi(\vec{a})}\subseteq \neg\name{\psi(\vec{a})}$, we have found $M\in (\calX_\calT)_\Obj$ with $\vec{a} \in \llbracket\vec{x}.\varphi \rrbracket_\Obj(M)$ but $\vec{a}\not\in \llbracket\vec{x}.\psi \rrbracket_\Obj(M)$.
\end{proof}

\section{Applications}

\subsection{Completeness of classical first-order logic}

As an application, we show how to encode classical first-order logic into propositional $\sigma$-coherent logic and prove a completeness theorem. This is essentially Rasiowa and Sikorski's proof \cite{RS51} of G\"{o}del's completeness theorem (see also Section~7 of \cite{S08}).

Fix a language $\calL$ consisting of countably many predicate symbols, each with a given arity.  First-order logic has logical symbols $\bot$, $\top$, $=$, $\neg$, $\wedge$, $\vee$, $\Rightarrow$, $\exists$, $\forall$, and countably many variable symbols $(x_k)_{k\in\IN}$. The formulas of $\calL$ are defined as usual. We omit function symbols for simplicity, but they can be encoded as predicate symbols as usual.

Let $T$ be a countable first-order theory in the language $\calL$. Define a propositional variable $\token{\varphi}$ for each $\calL$ formula $\varphi$ (possibly containing free variables), and let $V$ be the set of all such propositional variables. Let $\calT$ be the countable propositional $\sigma$-coherent theory with the following axiom schemas:
\begin{enumerate}
\item
$\token{\bot}\vdash \bot$
\item
$\top\vdash \token{\top}$
\item
$\token{\varphi}\wedge\token{\psi}\vdash \token{\varphi\wedge\psi}$
\item
$\token{\varphi\vee\psi}\vdash \token{\varphi}\vee\token{\psi}$
\item
$\token{\exists x_i.\varphi}\vdash \bigvee_{k\in\IN} \token{\varphi[x_k/x_i]}$
\item
$\token{\varphi}\vdash \token{\psi}$ (whenever $\varphi \Rightarrow \psi$ is $T$-provable using classical first-order logic)
\end{enumerate}
In 5, we assume all bound occurrences of $x_k$ in $\varphi$ are replaced by a new variable $x_\ell$ not occurring in $\varphi$ before doing the substitution $\varphi[x_k/x_i]$. The value of $\token{\varphi[x_k/x_i]}$ is independent of the choice of $x_\ell$ because of 6. It also follows from 6 that $\calT$ includes all classical axioms such as $\top \vdash \token{\varphi}\vee\token{\neg\varphi}$, $\token{\neg \exists x. \neg \varphi} \vdash \token{\forall x.\varphi}$, $\token{\varphi\Rightarrow\psi} \vdash \token{\neg\varphi\vee\psi}$, etc.

\begin{lemma}
If $\varphi$ and $\psi$ are $\calL$-formulas, then $\token{\varphi}\vdash^\calT \token{\psi}$ if and only if $\varphi\Rightarrow\psi$ is $T$-provable using classical first-order logic. 
\qed
\end{lemma}
See \cite{RS51} or Theorem~7.2 of \cite{S08} for the proof. Our approach is slightly different though. They use the fact that the Lindenbaum algebra of $T$ is a Boolean algebra, which corresponds via Stone duality to the clopen subsets of a Stone space. In Section~10 of \cite{S08}, the MacNeille completion is used to obtain a complete Heyting Algebra (i.e., a frame).  In contrast, we simply embed the Lindenbaum algebra of $T$ into the frame of opens of its Stone space, which is in fact the space $X_\calT$ of $\calT$-models.

Define a continuous map $f\colon X_\calT \to (\calS_\calL)_\Obj$ as follows (for $M\in X_\calT$):
\begin{itemize}
\item
$j \equiv_{f(M)} k$ if and only if $M \in \llbracket\token{x_j = x_k}\rrbracket^\calT$.
\item
$\llbracket P \rrbracket_{f(M)}(k_1,\ldots,k_n)$ holds if and only if $M \in \llbracket \token{P(x_{k_1},\ldots,x_{k_n}) }\rrbracket^\calT$.
\end{itemize}
It is easy to see that if $\varphi$ is a first-order $\calL$-sentence then $f(M) \vDash \varphi$ if and only if $M\in\llbracket \token{\varphi}\rrbracket^\calT$. The key part of the proof is that $\llbracket \token{\exists x_i.\varphi}\rrbracket^\calT \subseteq \bigcup_{k\in\IN} \llbracket \token{\varphi[x_k/x_i]}\rrbracket^\calT$, hence if $f(M) \vDash \exists x_i.\varphi$ then there is $k\in\IN$ such that $f(M)\vDash \varphi(k)$. In particular, $f(M)$ is a model of $T$ for each $M\in X_\calT$.

\begin{theorem}[G\"{o}del's Completeness Theorem]
If $T$ is a countable first-order theory, then every set theoretical model of $T$ satisfies $\varphi$ if and only if $\varphi$ is provable from $T$ in classical first-order logic.
\end{theorem}
\begin{proof}
If $\varphi$ is not provable from $T$ then $\top\not\vdash^\calT \token{\varphi}$, hence there is $M\in X_\calT$ such that $M\not\in \llbracket\token{\varphi}\rrbracket^\calT$. Then $f(M)$ is a model of $T$ and $f(M) \not\vDash \varphi$. 
\end{proof}

Since each first-order sentence $\varphi$ determines the clopen subset $\{M\in X_\calT \mid f(M)\vDash \varphi\}$ of $X_\calT$, the continuous map $f\colon X_\calT \to (\calS_\calL)_\Obj$ can be viewed as a Boolean-valued model. We investigate the more general Heyting-valued models in the next subsection.

\subsection{Heyting-valued models}

The following is a modification of Heyting-valued sets (see \cite{FS79}, C1.3.3 of \cite{J02},  and Definition~9.1 of \cite{S08}), where we have added countability assumptions.

\begin{definition}
Assume $A$ is a countably presented frame. An $A$-set is a countable set $M$ and a mapping $e\colon M\times M \to A$ such that for all $a,b,c\in M$:
\begin{itemize}
\item
$e(a,b)=e(b,a)$, and
\item
$e(a,b)\wedge e(b,c) \leq e(a,c)$.
\end{itemize}
An $n$-ary predicate on $M$ is a mapping $P\colon M^n\to A$ where
\[e(a_1,b_1)\wedge\cdots\wedge e(a_n,b_n)\wedge P(a_1,\ldots,a_n)\leq P(b_1,\ldots,b_n)\]
holds for all $a_1,\ldots,a_n,b_1,\ldots,b_n\in M$. $P$ is \emph{strict} if furthermore
\[P(a_1,\ldots,a_n) \leq e(a_1,a_1)\wedge\cdots\wedge e(a_n,a_n)\]
for all $a_1,\ldots,a_n\in M$.
\qed
\end{definition}

Let $X$ be quasi-Polish, and $f\colon X\to\ODS_\Obj$ a continuous map. Set 
\[M=\{a\in\IN \mid \text{$a\equiv a$ in $f(x)$ for some $x\in X$}\}.\]
For $a,b\in\IN$, define $e(a,b)$ to be the open subset of all $x\in X$  such that $a\equiv b$ in $f(x)$ (i.e, $e(a,b) = f^{-1}(\{ C\in\ODS_\Obj \mid a \equiv_C b\})$). Then $(M,e)$ is an $\PO(X)$-set. Similarly, a \emph{strict} $n$-ary predicate on $M$ is a continuous function $P\colon X\to\ODS_\Obj$ such that $P(x)$ is a subobject of $f(x)^n$ for each $x\in X$. Therefore, if we have a fixed countable language $\calL$, a continuous function from $X$ to $(\calS_\calL)_\Obj$ determines a $\PO(X)$-set $M$ and strict predicates on $M$ for each $P\in\calL$.

We next show that conversely, each $A$-set arises from a continuous function in this way. Assume $(M,e)$ is an $A$-set. We can assume $M= \IN$ by first embedding $M\hookrightarrow \IN$ and then extending $e$ to all of $\IN\times\IN$ by setting it to $\bot$ on pairs outside of $M\times M$. Furthermore, we can assume $A=\PO(X)$ for some quasi-Polish space $X$.

Now consider the propositional $\sigma$-coherent theory $\calT$, with propositional variables $V=\{\langle a,b\rangle \mid a,b\in\IN\}$ and axioms
\begin{itemize}
\item
$\langle a,b\rangle \vdash \langle b,a\rangle$
\item
$\langle a,b\rangle \wedge \langle b,c\rangle \vdash \langle a,c\rangle$
\end{itemize}
for each $a,b,c\in\IN$. This is the theory of partial equivalence relations on $\IN$, and it is easy to see that $\ODS_\Obj$ is the space of models of $\calT$, hence $\PO(\ODS_\Obj)$ is the frame presented by $\calT$. Since $e\colon V \to \PO(X)$ preserves the relations in $\calT$, Heckmann's spatiality result implies $e$ extends uniquely to a frame homomorphism $\hat{e}\colon\PO(\ODS_\Obj) \to \PO(X)$. Let $f\colon X\to\ODS_\Obj$ be the corresponding continuous function. Then for each $x\in X$ and $a,b\in\IN$, we have $x\in e(a,b)$ if and only if $a\equiv b$ in $f(x)$.

This extends to strict predicates as follows. Set
\[V' = V\cup\{ P_{a_1,\ldots, a_n} \mid P\in\calL \text{ and } a_1,\ldots,a_n \in\IN\},\]
where each $P_{a_1,\ldots, a_n}$ is a single new symbol (generator). Extend $\calT$ to $\calT'$ by adding the axioms
\begin{itemize}
\item
$\langle a_1,b_1 \rangle\wedge\cdots\wedge \langle a_n,b_n \rangle \wedge P_{a_1,\ldots,a_n}\vdash P_{b_1,\ldots,b_n}$
\item
$P_{a_1,\ldots,a_n} \vdash \langle a_1,a_1 \rangle\wedge\cdots\wedge \langle a_n,a_n\rangle$
\end{itemize}
for each $n$-ary $P\in\calL$ and $a_1,\ldots,a_n,b_1,\ldots,b_n\in M$. Then the space of models of $\calT'$ is homeomorphic to $(\calS_\calL)_\Obj$. Therefore, an $A$-set $(M,e)$ together with predicates on $(M,e)$ for each $P\in\calL$ determines a continuous function $f\colon X \to (\calS_\calL)_\Obj$ such that $f(x) \vDash P(a_1,\ldots,a_n)$ if and only if $x \in P(a_1,\ldots,a_n)$.

Non-strict predicates could be handled similarly if we weakened the requirement on $\calS_\calL$ that $\llbracket P\rrbracket_M$ be a subobject of $\overline{M}^n$. See Chapter~II of \cite{FS79} or Section~9 of \cite{S08} for applications to intuitionistic first-order logic.

\subsection{Forcing}

We briefly make some connections with forcing, as it is described in Section~16.D of Kechris \cite{ke95}. We use the fact from \cite{DPS} (see also \cite{dbr20}) that every quasi-Polish space is homeomorphic to the space of ideals $\I{\prec}$ of a transitive relation $\prec$ on $\IN$. $\I{\prec}$ has basic opens $[p]_\prec = \{I\in\I{\prec} \mid p\in I\}$ for $p\in\IN$.  Intuitively, the relation $\prec$ is an ``information ordering'' in the sense that $p\prec q$ means $q$ contains more information than $p$. However, $\prec$ is only required to be transitive, so it might not even be a preorder. For notational convenience, we sometimes write $q \succ p$ instead of $p \prec q$.

Fix a countable language $\calL$ (not necessarily set theory). We assume that we have an existence predicate $E(x)$ which is interpreted as $x \equiv_M x$ in each structure $M\in (\calS_\calL)_\Obj$ (see \cite{FS79,S08} for details, but note that we assume all predicates are strict).  Let $\prec$ be a transitive relation on $\IN$, and assume for simplicity that $p \prec q$ implies $[q]_\prec\not=\emptyset$. 

Assume we have a relation $p \Vdash_0 \alpha(\vec{a})$ defined on  tuples consisting of a number $p\in\IN$, an atomic formula $\vec{x}.\alpha$, and a compatible $\vec{a}\in\IN$. Further assume that $\vDash_0$ is monotonic in the sense that if $p \Vdash_0 \alpha(\vec{a})$ and $p\prec q$ then $q \Vdash_0 \alpha(\vec{a})$.

Define $f\colon\I{\prec}\to(\calS_\calL)_\Obj$ so that $f(I)$ is the smallest element of  $(\calS_\calL)_\Obj$ (with respect to the specialization order) such that 
\[f(I) \vDash \alpha(\vec{a}) \iff (\exists p\in I)\, p\Vdash_0 \alpha(\vec{a})\]
for each atomic $\alpha$ and compatible $a\in\IN^*$. It is easy to see that $f$ is well-defined and continuous.

Next, define the forcing relation $\Vdash$ on first-order formulas as follows:
\begin{itemize}
\item
$p \not\Vdash\bot$ and $p\Vdash\top$.
\item
$p\Vdash \alpha(\vec{a})$ if and only if $(\forall q \succ p)(\exists r \succ q)\, r \Vdash_0 \alpha(\vec{a})$, for atomic $\alpha$.
\item
$p\Vdash (\neg\varphi)(\vec{a})$ if and only if $(\forall q \succ p)\, q \not\Vdash \varphi$.
\item
$p\Vdash (\varphi \wedge\psi)(\vec{a})$ if and only if $p\Vdash\varphi(\vec{a})$ and $p\Vdash \psi(\vec{a})$.
\item
$p\Vdash (\varphi \vee \psi)(\vec{a})$ if and only if $(\forall q \succ p)(\exists r \succ q)(r\Vdash\varphi \text{ or } r\Vdash\psi)$.
\item
$p\Vdash (\varphi \Rightarrow \psi)(\vec{a})$ if and only if $(\forall q \succ p)(\exists r \succ q)(r\Vdash \neg\varphi \text{ or } r\Vdash\psi)$.
\item
$p\Vdash (\forall x.\varphi)(\vec{a})$ if and only if $(\forall b\in\IN)\, p \Vdash (E(x) \Rightarrow \varphi)(\vec{a},b)$.
\item
$p\Vdash (\exists x.\varphi)(\vec{a})$ if and only if $(\forall q \succ p)(\exists r \succ q)(\exists b\in\IN)\,p\Vdash (E(x)\wedge\varphi)(\vec{a},b)$.
\end{itemize}

Define the relation $M \vDash \varphi(\vec{a})$ as usual for classical first-order semantics. Note that for a (non-basic) first-order formula $\varphi$, the set $\{ M \mid M \vDash \varphi(\vec{a})\}$ is Borel but not necessarily open in $(\calS_\calL)_\Obj$. We have the following ``Truth Lemma'' for the forcing relation. 

\begin{lemma}
Assume $p\in\IN$,  $\vec{x}.\varphi$ is a first-order formula, and $\vec{a}\in\IN^*$ is compatible with $\vec{x}$. Then 
\[p\Vdash \varphi(\vec{a}) \iff f^{-1}(\{ M \mid M \vDash \varphi(\vec{a})\}) \text{ is comeager in } [p]_\prec.\]
In particular, for the ``generic'' $I\in\I{\prec}$ (i.e., for comeager many $I\in\I{\prec}$) and every $\varphi(\vec{a})$,
\[f(I)\vDash \varphi(\vec{a}) \iff (\exists p\in I)\, p\Vdash \varphi(\vec{a}).\] 
\end{lemma}
\begin{proof}
We prove the first claim for atomic $\alpha$. Assume $p\Vdash \alpha(\vec{a})$. We show that $f^{-1}(\{ M \mid M \vDash \alpha(\vec{a})\})$ is dense in $[p]_\prec$, hence it is comeager because it is open. If $U\subseteq\I{\prec}$ is open and $U\cap[p]_\prec\not=\emptyset$, then there is $q\succ p$ such that $\emptyset\not=[q]_\prec \subseteq U$.  The assumption $p\Vdash \alpha(\vec{a})$ implies there is $r_0\succ q$ such that $r_0 \Vdash_0 \alpha(\vec{a})$. Transitivity implies $r_0\succ p$, so repeating the argument with $r_0$ replacing $q$ yields $r_1 \succ r_0$ with $r_1\Vdash_0 \alpha(\vec{a})$. Continuing in this way we obtain a $\prec$-increasing sequence $(r_i)_{i\in\IN}$. Set $I=\{r\in\IN\mid (\exists i\in\IN) r_i \prec r\}$. Then $I\in U\cap[p]_\prec$ and $f(I)\vDash \alpha(\vec{a})$. Therefore, $f^{-1}(\{ M \mid M \vDash \alpha(\vec{a})\})$ is dense in $[p]_\prec$.

Conversely, assume $f^{-1}(\{ M \mid M \vDash \alpha(\vec{a})\})$ is comeager, hence dense, in $[p]_\prec$. If $q\succ p$ then  $[q]_\prec\subseteq [p]_\prec$, and the assumption on $\prec$ implies $[q]_\prec\not=\emptyset$, so there is $I\in [q]_\prec$ with $f(I)\vDash\alpha(\vec{a})$. Fix $s\in I$ such that $s\Vdash_0\alpha(\vec{a})$. Since $I$ is directed there is $r\in I$ with $q\prec r$ and $s\prec r$, and the motonicity of $\vDash_0$ implies $r\Vdash_0 \alpha(\vec{a})$.

Thus the first claim holds for atomic formulas. Arbitrary first-order formulas are handled by induction using Baire category methods (see Section~8.G of \cite{ke95}).

For the second claim, set $A = \{ I\in\I{\prec} \mid f(I) \vDash \varphi(\vec{a})\}$. Let $U(A)$ be the set of all $I\in\I{\prec}$ such that there is an open neighborhood $U$ of $I$ with $A$ comeager in $U$. Using the first half of this lemma and the fact that sets of the form $[p]_\prec$ form a basis for $\I{\prec}$, we have $I \in U(A)$ if and only if $(\exists p\in I)\, p\Vdash\varphi(\vec{a})$. By Theorem~8.29 of \cite{ke95}, the symmetric difference $(A\setminus U(A)) \cup (U(A)\setminus A)$ is meager, which completes the proof of the lemma.
\end{proof}

\section{Sober categories}

Readers familiar with \cite{Chen19} will notice that Chen works with the quasi-Polish \emph{groupoid} of models $\calG_\calT$, which is the subcategory of $\calX_\calT$ whose morphisms are restricted to be model isomorphisms, rather than arbitrary homomorphisms\footnote{Our $\calG_\calT$ corresponds to the groupoid from Example~6.9 in \cite{Chen25}, which is a slight modification of the groupoid of models in \cite{Chen19}. This modification extends the completeness result of \cite{Chen19} so that inequality ($\not=$) is not needed as a basic symbol.}. Technically, \cite{Chen19} shows that $\ODS^{\calG_\calT}$ is the $\sigma$-pretopos presented by $\calT$. In this section, we justify our claim in the introduction that $\ODS^{\calX_\calT}$ is presented by $\calT$ by showing that $\ODS^{\calX_\calT}$ and $\ODS^{\calG_\calT}$ are isomorphic $\sigma$-pretoposes (even topologically). 

Every continuous functor from $\calX_\calT$ to $\ODS$ can have its domain restricted to $\calG_\calT$, but going the other direction is slightly more complicated. Given a continuous functor $F\colon \calG_\calT \to \ODS$, its extension $\hat{F}\colon \calX_\calT \to \ODS$ can behave like $F$ on objects. The behavior of $\hat{F}$ on morphisms needs more care, but a natural extension is possible because the specialization order on $(\calG_\calT)_\Obj$ contains information about quotient objects and subobjects (as defined in Section~\ref{subsec_CatOfMods}), and continuous functors must preserve the specialization order. 

Let $h\colon A\to B$ be a homomorphism in $\calX_T$. For $a\in A$ and $b\in B$ (recall this means $a \equiv_A a$ and $b \equiv_B b$), we write $h(a)=b$ to mean that $h$ sends the equivalence class of $a$ to the equivalence class of $b$. Define $A'$ as follows:
\begin{itemize}
\item
$a \equiv_{A'} b$ if and only if $a,b\in A$ and $h(a) = h(b)$,
\item
$\llbracket R \rrbracket_{A'}(b_1, ... , b_n)$ if and only if there exist $a_1,...,a_n  \in A$ such that $a_i \equiv_{A'} b_i$ for each $i \leq n$ and $\llbracket R\rrbracket_A(a_1,...,a_n)$.
\end{itemize}
Clearly $\equiv_A \subseteq \equiv_{A'}$ and $\llbracket R \rrbracket_A \subseteq \llbracket R \rrbracket_{A'}$ for each predicate symbol $R$, hence $A \leq A'$, where $\leq$ is the specialization order on $\calX_\calT$. Next, define $B'$ as follows:
\begin{itemize}
\item
$b \equiv_{B'} c$ if and only if $b \equiv_B c$ and there is $a \in A$ such that $h(a)=b$,
\item
$\llbracket R \rrbracket_{B'}(b_1, ... , b_n)$ if and only if there exist $a_1,...,a_n  \in A$ such that $h(a_i)= b_i$ for each $i \leq n$ and $\llbracket R\rrbracket_A(a_1,...,a_n)$.
\end{itemize}
Clearly $\equiv_{B'}\subseteq\equiv_B$, and $\llbracket R \rrbracket_{B'} \subseteq \llbracket R \rrbracket_{B}$ holds for each predicate symbol $R$ because $h$ is a homomorphism, hence $B' \leq B$. There is an obvious isomorphism $i\colon A' \cong B'$ induced by $h$, hence we obtain a factorization of $h$ as $A \leq A' \cong B' \leq B$. Since $F$ is continuous, $F_\Obj(A) \leq F_\Obj(A')$ and $F_\Obj(B') \leq F_\Obj(B)$ in $\ODS_\Obj$, and there are morphisms $e \colon F_\Obj(A) \to F_\Obj(A')$ and $e' \colon F_\Obj(B') \to F_\Obj(B)$ in $\ODS_\Mor$ corresponding to these embeddings. Therefore, we obtain a functor $\hat{F}$ by defining $\hat{F}_\Mor(h\colon A\to B)$ as the composition $e'\circ F_\Mor(i) \circ e$. The mapping $F \to \hat{F}$ from $(\ODS^{\calX_\calT})_\Obj$ to $(\ODS^{\calG_\calT})_\Obj$ is continuous (assuming the functor categories have the compact-open topologies), hence $(\ODS^{\calX_\calT})_\Obj$ and $(\ODS^{\calG_\calT})_\Obj$ are homeomorphic.

Next we must show that $(\ODS^{\calX_\calT})_\Mor$ and $(\ODS^{\calG_\calT})_\Mor$ are homeomorphic. An element of $(\ODS^{\calX_\calT})_\Mor$ is a tuple $\langle \eta, F, G\rangle$ where $F,G\colon \calX_\calT \to \ODS$ are continuous functors and $\eta\colon (\calX_\calT)_\Obj \to \ODS_\Mor$ is a continuous function determining a natural transformation from $F$ to $G$. $(\ODS^{\calG_\calT})_\Mor$ is defined similarly. Again, going from $(\ODS^{\calX_\calT})_\Mor$ to $(\ODS^{\calG_\calT})_\Mor$ is easy. In the other direction, since $(\calX_\calT)_\Obj$ and $(\calG_\calT)_\Obj$ are equal, we only need to check that if $\langle \eta, F, G\rangle$ is in $(\ODS^{\calG_\calT})_\Mor$ then $\eta$ is still a natural transformation from $\hat{F}$ to $\hat{G}$. This requires showing that if $h\colon A\to B$ is a homomorphism, then $\eta(B) \circ \hat{F}(h) = \hat{G}(h) \circ \eta(A)$. Now if $A \leq A'$ then the continuity of $\eta$ implies $\eta(A) \leq \eta(A')$ in $\ODS_\Mor$, which means the graph of the function $\eta(A)$ is included in the graph of $\eta(A')$, hence the composition $F(A) \stackrel{\eta_{A}}{\to} G(A)\leq G(A')$ equals $F(A) \leq F(A') \stackrel{\eta_{A'}}{\to} G(A')$. It then follows from the definition of $\hat{F}$ and the naturality of $\eta$ (with respect to isomorphisms) that $\eta(B) \circ \hat{F}(h) = \hat{G}(h) \circ \eta(A)$. Therefore, $(\ODS^{\calX_\calT})_\Mor$ and $(\ODS^{\calG_\calT})_\Mor$ are homeomorphic.

It follows that $\ODS^{\calX_\calT}$ and $\ODS^{\calG_\calT}$ are (topologically) isomorphic $\sigma$-pretoposes. In many applications it would probably be more convenient to use the groupoid $\calG_\calT$ instead of the category $\calX_\calT$. However, we next show that if we follow the analogy between spaces and toposes then $\calX_\calT$ should be viewed as the ``sobrification'' of $\calG_\calT$. 

To see this, first recall how sobriety is defined for spaces. If $X$ is a space, then $X$ is sober precisely when the embedding $X \hookrightarrow \sierp^{\sierp^X}$, defined as $x \mapsto \lambda f.f(x)$, is a bijection between the points of $X$ and the frame homomorphisms from $\sierp^X$ to $\sierp$. A frame homomorphism is a (continuous) function that preserves arbitrary joins (colimits) and finite meets (limits). For countably based spaces, it suffices that it preserves countable joins rather than arbitrary joins.

Similarly, we show that $\calX_\calT$ is ``sober'' in the sense that it is equivalent to the category of continuous functors $\calF\colon \ODS^{\calX_\calT} \to \ODS$ that preserve countable colimits and finite limits. We call a continuous functor between topological $\sigma$-pretoposes that preserves countable colimits and finite limits a \emph{$\sigma$-pretopos morphism}.

First, we define a functor\footnote{We use $\y$ because of the similarities with the double powerspace monad.} $\y$ from $\calX_\calT$ to $\ODS^{\ODS^{\calX_\calT}}$, the category of continuous functors from $\ODS^{\calX_\calT}$  to $\ODS$. Formally, we will work in the cartesian closed category $\QCB$, and then it will be clear that $\y$ will be a continuous functor. Readers unfamiliar with $\QCB$ can assume the compact-open topology on all function spaces (this follows from the consonance of quasi-Polish spaces, but we omit the details).

The functor $\y_\Obj(M)\colon \ODS^{\calX_\calT} \to \ODS$ is defined for each object $M$ of $\calX_\calT$ as follows:
\begin{itemize}
\item
$(\y_\Obj(M))_\Obj(F)=F_\Obj(M)$ for each continuous functor $F\colon \calX_\calT \to \ODS$, and
\item
$(\y_\Obj(M))_\Mor(\eta\colon F\to G)$ is defined as $\eta(M)\colon F_\Obj(M)\to G_\Obj(M)$  for each continuous natural transformation $\eta \colon F\to G$. 
\end{itemize}
The natural transformation $\y_\Mor(h)\colon \y_\Obj(M) \to \y_\Obj(M')$ is defined for each homomorphism $h\colon M\to M'$ in $\calX_\calT$ as $\y_\Mor(h)(F)=F_\Mor(h)$. This is a natural transformation because the diagram on the left commutes for every natural transformation $\eta\colon F\to G$ (on the right is the same diagram but relabeled).
\begin{center}
\begin{tikzcd}
F_{\Obj}(M) \arrow{r}{F_{\Mor}(h)}\arrow[swap]{d}{\eta(M)}  & F_{\Obj}(M') \arrow{d}{\eta(M')} \\
G_{\Obj}(M) \arrow{r}{G_{\Mor}(h)} & G_{\Obj}(M')
\end{tikzcd}
\qquad
\begin{tikzcd}
\y(M)(F) \arrow{r}{\y(h)(F)}\arrow[swap]{d}{\y(M)(\eta)}  & \y(M')(F) \arrow{d}{\y(M')(\eta)} \\
\y(M)(G) \arrow{r}{\y(h)(G)} & \y(M')(G)
\end{tikzcd}
\end{center}

We show that $\y$ is full and faithful.

Let $h,h'\colon M\to M'$ be morphisms in $\calX_\calT$ such that $\y_\Mor(h)=\y_\Mor(h')$. Then \[\overline{h}=\llbracket x.\top\rrbracket_\Mor(h) = \y_\Mor(h)(\llbracket x.\top\rrbracket^\calT) =  \y_\Mor(h')(\llbracket x.\top\rrbracket^\calT) =\overline{h'},\]
hence $h=h'$. Therefore, $\y$ is faithful.

Next we show $\y$ is full. Let $\varepsilon\colon \y_\Obj(M) \to \y_\Obj(M')$ be a continous natural transformation. Setting $\overline{h} = \varepsilon(\llbracket x.\top\rrbracket^\calT)$, we have that $\overline{h}\colon \overline{M}\to\overline{M'}$ is a morphism of $\ODS$. We must show it is a homomorphism.

First note that if $\vec{x}$ is a context of $n$ variables, then for each $i\leq n$ there is a natural transformation $\eta_i \colon \llbracket \vec{x}.\top \rrbracket^\calT \to \llbracket x_i.\top\rrbracket^\calT$ defined for each $\calT$-model $M''$ as the $i$th projection $\pi_i\colon (\overline{M''})^n \to \overline{M''}$. Thus $\y_\Obj(M)(\eta_i) \colon \y_\Obj(M)(\llbracket \vec{x}.\top\rrbracket^\calT) \to \y_\Obj(M)(\llbracket x_i.\top\rrbracket^\calT)$ is the projection $\pi_i\colon \overline{M}^n \to \overline{M}$, hence the following diagram commutes for each $i\leq n$ by the naturality of $\varepsilon$ and the definition of $\overline{h}$:
\begin{center}
\begin{tikzcd}[column sep=huge]
(\overline{M})^n \arrow{r}{\varepsilon(\llbracket \vec{x}.\top\rrbracket^\calT)}\arrow[swap]{d}{\pi_i}  & (\overline{M'})^n \arrow{d}{\pi_i} \\
\overline{M} \arrow{r}{\overline{h}} & \overline{M'}
\end{tikzcd}
\end{center}
It follows that $\overline{h}^n=\varepsilon(\llbracket \vec{x}.\top\rrbracket^\calT)$. 

Next, for each $\calL$-formula $\vec{x}.\varphi$ there is a natural transformation $\eta_\varphi \colon \llbracket \vec{x}.\varphi \rrbracket^\calT \to \llbracket \vec{x}.\top\rrbracket^\calT$ defined for each $M''$ as the embedding $\llbracket \vec{x}.\varphi \rrbracket^\calT(M'') \leq (\overline{M''})^n$, and we obtain the following commutative diagram:
\begin{center}
\begin{tikzcd}[column sep=huge]
\llbracket \vec{x}.\varphi\rrbracket^\calT(M) \arrow{r}{\varepsilon(\llbracket \vec{x}.\varphi\rrbracket^\calT)}\arrow[d, phantom, sloped, "\leq"]  & \llbracket \vec{x}.\varphi\rrbracket^\calT(M') \arrow[d, phantom, sloped, "\leq"] \\
(\overline{M})^n \arrow{r}{\overline{h}^n} & (\overline{M'})^n
\end{tikzcd}
\end{center}
In particular, $\llbracket P\rrbracket_M(a_1,\ldots,a_n)$ implies $\llbracket P\rrbracket_{M'}(\overline{h}(a_1),\ldots,\overline{h}(a_n))$ for each $P\in\calL$, hence $\overline{h}$ determines a morphism $h\colon M\to M'$.

It remains to show that $\y_\Mor(h)=\varepsilon$. This amounts to showing that $\varepsilon(F)=F_\Mor(h)$ for each $F\colon \calX_\calT\to\ODS$. It follows from the previous commutative diagram that $\varepsilon(\llbracket \vec{x}.\varphi\rrbracket^\calT) = \llbracket \vec{x}.\varphi\rrbracket^\calT(h)$ for each $\calL$-formula $\varphi$.

Now let $\calD$ be a countable diagram in $\ODS^{\calX_\calT}$, and assume $\varepsilon(F)=F_\Mor(h)$ for each $F$ in $\calD$. Let $L$ be the colimit of $\calD$ with cocone given by the morphism (i.e., natural transformation) $\eta_F \colon F\to L$ for each $F$ in $\calD$. Since colimits in $\ODS^{\calX_\calT}$ are computed objectwise, $L_\Obj(M)$ is the colimit of the diagram $\calD_M$ in $\ODS$ obtained by evaluating everything in $\calD$ at $M$ (i.e., $\calD_M$ is the image of $\calD$ under the functor $\y_\Obj(M)$). The family of morphisms $\eta_F(M')\circ \varepsilon(F) \colon F(M) \to L(M')$ (for $F$ in $\calD$) form a cocone of $\calD_M$, and $\varepsilon(L)\colon L_\Obj(M)\to L_\Obj(M')$ satisfies $\varepsilon(L)\circ \eta_F(M) = \eta_F(M')\circ\varepsilon(F)$. But $\varepsilon(F)=F_\Mor(h)$ by assumption, and $L_\Mor(h)\circ \eta_F(M) = \eta_F(M')\circ F_\Mor(h)$ by the naturality of $\eta_F$. Therefore, $\varepsilon(L)=L_\Mor(h)$ by the universal property of colimits.

The dual version of the above argument shows that if $\calD$ is a finite diagram in $\ODS^{\calX_\calT}$ and $\varepsilon(F)=F_\Mor(h)$ for each $F$ in $\calD$, then the limit $L$ of $\calD$ satisfies $\varepsilon(L)=L_\Mor(h)$.

By Theorem~8.1 of \cite{Chen19}, reinterpreted as a result about functors by using \cite{dbr26}, every continuous functor $F\colon \calX_\calT \to\ODS$ can be constructed from functors of the form $\llbracket \vec{x}.\varphi\rrbracket^\calT$ by using countable colimits and finite limits. Therefore, $\varepsilon(F)=F_\Mor(h)$ for each $F\colon \calX_\calT\to\ODS$, hence $\y_\Mor(h)=\varepsilon$. This completes the proof that $\y$ is full.

It is easy to see that $\y_\Obj(M)$ is a $\sigma$-pretopos morhism because colimits and limits are defined objectwise. We show that conversely every $\sigma$-pretopos morphism is isomorphic to a functor in the range of $\y$. Let $\calF\colon \ODS^{\calX_\calT} \to \ODS$ be a $\sigma$-pretopos morphism. Set $\overline{M}_\calF = \calF_\Obj(\llbracket x.\top\rrbracket^\calT)$, and for $n$-ary $P\in\calL$ fix a context $\vec{x}$ with $n$ variables and define $\llbracket P \rrbracket_{M_\calF} = \calF_\Obj(\llbracket \vec{x}.P(\vec{x})\rrbracket^\calT)$. Since for each $M$ we have $\llbracket \vec{x}.P(\vec{x})\rrbracket^\calT(M) \leq \llbracket \vec{x}.\top\rrbracket^\calT(M)$ (the specialization order on $\ODS$), we have $\llbracket \vec{x}.P(\vec{x})\rrbracket^\calT \leq \llbracket \vec{x}.\top\rrbracket^\calT$ (the specialization order on $\ODS^{\calX_\calT}$), hence\footnote{Technically, the equalities are only natural isomorphisms, but we omit the details for clarity.}
\[\llbracket P \rrbracket_{M_\calF} =  \calF_\Obj(\llbracket \vec{x}.P(\vec{x})\rrbracket^\calT) \leq  \calF_\Obj(\llbracket \vec{x}.\top\rrbracket^\calT)=  \calF_\Obj((\llbracket x.\top\rrbracket^\calT)^n) = (\overline{M}_\calF)^n,\]
because $\calF_\Obj$ preserves the specialization order and finite products. Therefore, $M_\calF=(\overline{M}_\calF,\llbracket\cdot\rrbracket_{M_\calF})$ is an $\calL$-structure.

Next, $\llbracket \vec{x}.x_i = x_j\rrbracket^\calT$ is the pullback of the projections $\pi_i,\pi_j\colon (\llbracket x.\top\rrbracket^\calT)^n \to \llbracket x.\top\rrbracket$ in $\ODS^{\calX_\calT}$, and $\calF_\Obj$ preserves finite limits, hence  $\calF_\Obj(\llbracket \vec{x}.x_i = x_j\rrbracket^\calT)$ is the pullback of the projections $\pi_i,\pi_j\colon (\overline{M}_\calF)^n \to \overline{M}_\calF$, which is $\llbracket \vec{x}.x_i = x_j\rrbracket_\Obj( M_\calF)$. Similarly (see Section~10 of \cite{Chen19}), 
\begin{itemize}
\item
$\llbracket \vec{x}.\varphi \wedge \psi\rrbracket^\calT$ is the pullback of $\llbracket \vec{x}.\varphi \rrbracket^\calT$ and $\llbracket \vec{x}.\psi\rrbracket^\calT$.
\item
$\llbracket \vec{x}.\bigvee_{i\in\IN} \varphi_i\rrbracket^\calT$ is a quotient of the coproduct of $\llbracket \vec{x}. \varphi_i\rrbracket^\calT$ ($i\in\IN$).
\item
$\llbracket \vec{x}.(\exists y)  \varphi\rrbracket^\calT$ is the image of the composition $\llbracket \vec{x},y.  \varphi\rrbracket^\calT \leq \llbracket \vec{x},y.\top\rrbracket^\calT\to \llbracket \vec{x}.\top\rrbracket^\calT$, where the last morphism is the projection.
\end{itemize}
Since $\calF_\Obj$ preserves the above structure, we obtain $\calF_\Obj(\llbracket \vec{x}.\varphi\rrbracket^\calT) = \llbracket \vec{x}.\varphi\rrbracket_\Obj( M_\calF)$ for each formula $\vec{x}.\varphi$.

If $\sigma \vdash_{\vec{y}} \theta$ is an axiom of $\calT$, then $\llbracket \vec{y}.\sigma \rrbracket^\calT \leq \llbracket \vec{y}.\theta \rrbracket^\calT$ and the continuity of $\calF_\Obj$ implies 
\[\llbracket \vec{y}.\sigma\rrbracket_\Obj( M_\calF) = \calF_\Obj(\llbracket \vec{y}.\sigma \rrbracket^\calT) \leq \calF_\Obj(\llbracket \vec{y}.\theta \rrbracket^\calT)=\llbracket \vec{y}.\theta\rrbracket_\Obj( M_\calF).\] 
Thus $M_{\calF}$ is a $\calT$-model. 

Finally, by another application of Theorem~8.1 of \cite{Chen19}, every continuous functor $F\colon \calX_\calT \to\ODS$ can be constructed from functors of the form $\llbracket \vec{x}.\varphi\rrbracket^\calT$ using countable colimits and finite limits, hence $\calF_\Obj(F) = F_\Obj(M_\calF)$. 

Next we show $\calF_\Mor(\eta)=\eta(M_\calF)$ for each natural transformation $\eta\colon F\to G$. Let $Gr(\eta)\colon \calX_\calT\to\ODS$ be the graph of $\eta$, which is the subobject of $F\times G$ corresponding to the image of the morphism $\langle 1_F, \eta\rangle \colon F\to F\times G$. Thus $\eta$ factors as the composition of the image $F\to Gr(\eta)$ with the subobject embedding $Gr(\eta)\leq F\times G$ and the projection $F\times G \to G$. Therefore, $\calF_\Mor(\eta)$ factors as 
\[F_\Obj(M_\calF)=\calF_\Obj(F)\to \calF_\Obj(Gr(\eta))\leq \calF_\Obj(F\times G)\to\calF_\Obj(G)=G_\Obj(M_\calF).\] 
But since $\calF$ preserves images and products, the above composition is the factorization of $\eta(M_\calF)$ through its graph $(Gr(\eta))_\Obj(M_\calF)=\calF_\Obj(Gr(\eta))$, which proves $\calF_\Mor(\eta)=\eta(M_\calF)$.

Therefore, $\y_\Obj(M_\calF) = \calF$ (up to natural isomorphism). This proves the following.

\begin{theorem}\label{thm:XTsober}
$\y\colon \calX_\calT\to \ODS^{\ODS^{\calX_\calT}}$ restricts to a continuous equivalence between $\calX_\calT$ and the category of $\sigma$-pretopos morphisms from $\ODS^{\calX_\calT}$ to $\ODS$.
\qed
\end{theorem}

Note that the equivalence is continuous in both directions because the mapping $\calF \mapsto M_\calF$ in the above proof is continuous. Therefore, $\calX_\calT$ is a ``sober'' quasi-Polish category which is completely determined by the topological $\sigma$-pretopos $\ODS^{\calX_\calT}$. Since $\ODS^{\calX_\calT}$ and $\ODS^{\calG_\calT}$ are isomorphic $\sigma$-pretoposes, and these are analogous to the ``frames'' of $\calX_\calT$ and $\calG_\calT$, it is natural to view $\calX_\calT$ as the ``sobrification'' of $\calG_\calT$. The analogy extends to the duality between frame homomorphisms and continuous functions as follows.

\begin{theorem}
Let $\calT_1$ and $\calT_2$ be countable $\sigma$-coherent theories over countable languages $\calL_1$ and $\calL_2$, respectively. There is a dual equivalence between continuous functors $F\colon \calX_{\calT_1} \to \calX_{\calT_2}$ and $\sigma$-coherent morphisms $\calF\colon \ODS^{\calX_{\calT_2}} \to \ODS^{\calX_{\calT_1}}$.
\end{theorem}
\begin{proof}
To simplify notation, set $\calC=\calX_{\calT_1}$ and $\calD=\calX_{\calT_2}$. Given $F\colon \calC\to\calD$ define $\calF_F\colon \ODS^\calD \to \ODS^\calC$ as:
\begin{itemize}
\item
$(\calF_F)_\Obj(G) = G\circ F$, for $G\colon \calD\to\ODS$.
\item
$(\calF_F)_\Mor(\eta\colon G\to H)$ is the natural transformation $\varepsilon \colon G\circ F \to H\circ F$ given by $\varepsilon(C)=\eta(F_\Obj(C))$ for each $C\in\calC_\Obj$.
\end{itemize}
Clearly $\calF_F$ is a $\sigma$-coherent morphism because colimits and limits are computed objectwise.

Next, given $\calF\colon \ODS^\calD \to \ODS^\calC$, first define $\widehat{\calF}\colon \calC\to \ODS^{\ODS^\calD}$ as:
\begin{itemize}
\item
$\widehat{\calF}_\Obj(C)\colon \ODS^{\calD}\to\ODS$ is defined for $C\in\calC_\Obj$ as:
\begin{itemize}
\item
$(\widehat{\calF}_\Obj(C))_\Obj(G) = (\calF_\Obj(G))_\Obj(C)$ for $G\colon \calD\to\ODS$.
\item
$(\widehat{\calF}_\Obj(C))_\Mor(\eta\colon G\to H) = \calF_\Mor(\eta)(C)$.
\end{itemize}
\item
$\widehat{\calF}_\Mor(f\colon C\to C')$ is the natural transformation $\varepsilon\colon \widehat{\calF}_\Obj(C)\to \widehat{\calF}_\Obj(C')$ defined as $\varepsilon(G) = (\calF_\Obj(G))_\Mor(f)$ for each $G\colon \calD\to\ODS$.
\end{itemize}
$\widehat{\calF}_\Obj(C)$ is a $\sigma$-coherent morphism for each $C\in\calC_\Obj$ because $\calF$ is a $\sigma$-coherent morphism and colimits and limits in $\ODS^\calD$ are computed objectwise. Therefore, we can compose $\widehat{\calF}$ with the equivalence from Theorem~\ref{thm:XTsober} to get a continuous functor $F_\calF\colon\calC\to\calD$.

Going from $F\colon \calC\to\calD$ to $\calF_F\colon\ODS^\calD \to \ODS^\calC$ to $\widehat{\calF_F}\colon \calC\to \ODS^{\ODS^\calD}$ as above, we see that $\widehat{\calF_F} = \y\circ F$. Applying the equivalence from Theorem~\ref{thm:XTsober} shows that $F_{\calF_F}$ is naturally isomorphic to $F$.

Conversely, go from $\calF\colon\ODS^\calD \to \ODS^\calC$ to $\widehat{\calF}\colon \calC\to \ODS^{\ODS^\calD}$ to $F_\calF\colon \calC\to\calD$ as above. $\widehat{\calF}$ and $\y\circ F_\calF$ are naturally isomorphic by definition of $F_\calF$. For $G\colon \calD\to\ODS$ and $C\in\calC_\Obj$ and $\eta\colon G\to H$ we have:
\begin{eqnarray*}
(\calF_{F_\calF})_\Obj(G)(C) &=& G_\Obj((F_\calF)_\Obj(C))\\
&=& ((\y\circ F_\calF)_\Obj(C))_\Obj(G)\\
&\cong& (\widehat{\calF}_\Obj(C))_\Obj(G)\\
&=& (\calF_\Obj(G))_\Obj(C)
\end{eqnarray*}
and
\begin{eqnarray*}
(\calF_{F_\calF})_\Mor(\eta)(C) &=& \eta((F_\calF)_\Obj(C))\\
&=& ((\y\circ F_\calF)_\Obj(C))_\Mor(\eta)\\
&\cong& (\widehat{\calF}_\Obj(C))_\Mor(\eta)\\
&=& \calF_\Mor(\eta)(C).
\end{eqnarray*}
Therefore, $\calF_{F_\calF}$ is naturally isomorphic to $\calF$.
\end{proof}

We can define $\y$ similarly for any topological category $\calC$ in $\QCB$, and call $\calC$ \emph{sober} (with respect to $\ODS$) if $\y\colon \calC \to \ODS^{\ODS^\calC}$ is an equivalence with the subspace of $\sigma$-pretopos morphisms. We conjecture that the ``sober'' quasi-Polish categories are exactly the spaces of models of countable predicate $\sigma$-coherent theories (up to continuous equivalence). 

The mapping $\calC \mapsto \ODS^{\ODS^\calC}$ should extend to a monad on the category of quasi-Polish categories, and then the theory could be developed in a way similar to Taylor's Abstract Stone Duality \cite{Taylor02}, but we leave this for future research.

It would also be interesting to define the embedding $\y$ using pretoposes other than $\ODS$, and investigate the corresponding notion of ``sobriety''. An interesting example of a quasi-Polish pretopos other than $\ODS$ is the quasi-Polish category $\CHS$ of compact Hausdorff quasi-Polish spaces. $\CHS$ has countable limits and finite colimits, but it does not have all countable colimits, hence it is not a $\sigma$-pretopos. What is the ``$\ODS$-sobrification'' of $\CHS$? What is the ``$\CHS$-sobrification'' of $\ODS$?

\section*{Acknowledgments}

The author thanks Ruiyuan Chen for helpful discussions and comments on an earlier version of this paper. This work was supported by JSPS Bilateral Program Number JPJSBP120265001.


\bibliographystyle{amsplain}
\bibliography{myrefs}


\end{document}